\documentclass[11pt,reqno]{amsart}
\usepackage[margin=1.4in,letterpaper]{geometry}
\usepackage{graphicx}
\usepackage{amssymb,amsmath,amsthm,xcolor,mathrsfs,stmaryrd,accents,enumitem,hyperref}
\usepackage{verbatim}
\usepackage{epstopdf}
\makeatletter\let\over\@@over\makeatother

\newcommand{\be}{\begin{equation} }
\newcommand{\ee}{\end{equation}}
\newcommand{\bse}{\begin{subequations}}
\newcommand{\ese}{\end{subequations}}

\numberwithin{equation}{section}
\theoremstyle{plain} 
\newtheorem{theorem}{Theorem}[section] 
\newtheorem{proposition}{Proposition}[section] 
\newtheorem{corollary}[theorem]{Corollary}
 
\newtheorem{lem}{Lemma}[section]

\newtheorem{definition}[theorem]{Definition}
\theoremstyle{remark}
\newtheorem{remark}{Remark}[section]

\newcommand{\LV}{\left|}
\newcommand{\RV}{\right|}

\newcommand{\LB}{\left[}
\newcommand{\RB}{\right]}
\newcommand{\LC}{\left(}
\newcommand{\RC}{\right)}

\newcommand{\R}{\mathbb{R}}

\usepackage{enumitem} % some list environment customization
\usepackage{hyperref}

\newcommand{\even}{\mathrm{e}}      % spaces of even functions
\newcommand{\bdd}{\mathrm{b}}       % spaces of bounded functions
\newcommand{\F}{{\mathscr F}}       % nonlinear operator
\newcommand{\cm}{{\mathscr C}}      % various continua 
\newcommand{\loc}{{\mathrm{loc}} }  % local
\newcommand\Xspace{\mathscr X}
\newcommand\Yspace{\mathscr Y}

\newcommand{\placeholder}{\;\cdot\;}

\newcommand{\n}[2][]{#1\lVert #2 #1\rVert}

\allowdisplaybreaks

\title[Solitary waves of the Boussinesq $abcd$ system]{Global bifurcation of solitary waves to the Boussinesq $abcd$ system}
\author[R. M. Chen]{Robin Ming Chen} 
\address{Department of Mathematics, University of Pittsburgh, Pittsburgh PA 15260}
\email{mingchen@pitt.edu}

\author[J. Jin]{Jie Jin}
\address{Department of Mathematics, University of Pittsburgh, Pittsburgh PA 15260}
\email{jij50@pitt.edu} 

\begin{document}

\begin{abstract}
The Boussinesq $abcd$ system arises in the modeling of long wave small amplitude water waves in a channel, where the four parameters $(a,b,c,d)$ satisfy one constraint. In this paper we focus on the solitary wave solutions to such a system. In particular we work in two parameter regimes where the system does not admit a Hamiltonian structure (corresponding to $b \ne d$). We prove via analytic global bifurcation techniques the existence of solitary waves in such parameter regimes. Some qualitative properties of the solutions are also derived, from which sharp results can be obtained for the global solution curves. 

Specifically, we first construct solutions bifurcating from the stationary waves, and obtain a global continuous curve of solutions that exhibits a loss of ellipticity in the limit. The second family of solutions bifurcate from the classical Boussinesq supercritical waves. We show that the curve associated to the second class either undergoes a loss of ellipticity in the limit or becomes arbitrarily close to having a stagnation point.  
\end{abstract}

\thispagestyle{empty}
\maketitle

\setcounter{tocdepth}{1}
\tableofcontents

%\noindent {\bf keywords.} 

%\noindent
% {\bf {AMS classifications.}}

%%%%%%%%%%%%%%%%%%%%%%%%%%%%%%%%%%%%%%%%%%%%%%%%%%%%%%%%%%%%%%%%%%%%%%%%%%%%%%%%%%%%%%%%%%%
\section{Introduction}\label{sec:intro}

The phenomenon of solitary wave was first observed by John Scott Russel \cite{Russell} almost two centuries ago, which is later used to characterize wave that does not disperse and retains its original identity as time evolves. Exact existence theory for solitary water waves, however, first appeared more than a century later in the work of Lavrentiev \cite{Lavrentiev1954}, Friedrichs-Hyers \cite{FriedrichsHyers1954}, and Ter-Krikorov \cite{Ter-Krikorov1960} for small-ampliltude irrotational waves. Construction for large-amplitude irrotational waves was achieved by Amick--Toland \cite{AT1,AT2} and Benjemin--Bona--Bose \cite{BBB}.

Russel's experiment also motivated the studies on the mathematical modeling of water waves. The first works can be dated back to Boussinesq \cite{Bous}, Rayleigh \cite{Rayleigh}, Korteweg and de Vries \cite{KdV}, where simpler sets of equations were derived as asymptotic models from the free surface Euler equations in some specific physical regimes. To be more precise, let $h$ and $\lambda$ denote respectively the mean elevation of the water over the bottom and the typical wavelength, and let $a$ be a typical wave amplitude. The parameter regime considered in the above works corresponds to 
\begin{equation*}%\label{regime}
\varepsilon = {a\over h} \ll 1, \quad \delta = {h \over \lambda} \ll 1, \quad \varepsilon = O(\delta^2),
\end{equation*}
which is called the small amplitude, shallow water regime. Physically, $\varepsilon$ measures the strength of nonlinearity while $\mu$ characterizes the effect of dispersion. Thus solitary waves can be viewed as generated from a perfect balance between nonlinear and dispersive effects. 
The reduced systems within the above scaling regime couple the free surface elevation $\eta$ to the horizontal component of the velocity $u$, and include the celebrated KdV equation and the Boussinesq equation \cite{Bous, Craig, KanoNishida, KdV, Ursell}.

In this paper we will consider solitary wave solutions to an asymptotic water wave model derived by 
%Taking advantage of the freedom associated with the choice of the velocity
%variable and making full use of the lower-order relations in the dispersive terms, 
Bona--Chen--Saut \cite{BCS1} (generalized to include the surface tension in \cite{Daripa_Dash} and in higher dimensions Bona--Colin--Lannes \cite{BonaColinLannes}) as an extended system of the classical Boussinesq equation. Specifically, it is a three-parameter family of Boussinesq systems for one dimensional surfaces that takes the following form
\begin{equation}\label{1dBoussinesq}
\left\{\begin{array}{l}
\eta_t + u_x + (u\eta)_x + a u_{xxx} - b \eta_{xxt} = 0, \\
u_t + \eta_x + {1\over2}(u^2)_x + c \eta_{xxx} - d u_{xxt} = 0,
\end{array}\right.
\end{equation}
all of which are formally equivalent models of solutions of the Euler equations.
In the above system $\eta$ is proportional to the deviation of the free surface from its rest position, ${u}$ is proportional to the horizontal velocity taken at the scaled height $0\leq \theta \leq 1$ ($\theta = 1$ at the free surface and $\theta = 0$ at the bottom). 
The parameters have the following explicit form
%Although it seems that \eqref{1dBoussinesq} involves four parameters, they indeed satisfy $a + b + c + d = {1\over 3} - \tau$, where $\tau\geq0$ is the normalized surface tension. In particular, they have the following explicit form
\begin{equation*}
\displaystyle a = \left({\theta^2\over2} - {1\over6} \right)\nu,\quad  \displaystyle b = \left({\theta^2\over2} - {1\over6} \right)(1 - \nu) ,\quad
\displaystyle c = {(1-\theta^2)\over2}\mu - \tau, \quad  \displaystyle d = {(1-\theta^2)\over2}(1 - \mu)
\end{equation*}
with $\nu$ and $\mu$ arbitrary real numbers, and $\tau\geq0$ is the normalized surface tension. These three degrees of freedom arise from the height at which the horizontal velocity is taken and from a double use of the {\it BBM trick} \cite{BBM}. The hydrodynamic relevance of the model was justified in \cite{BonaColinLannes, Chazel, SautXu}.

A solitary wave solution to system \eqref{1dBoussinesq} is of the type 
\begin{equation}\label{solitaryform}
\eta(x,t) = \eta(\xi) = \eta(x- \lambda t)\in H^1(\R), \quad u(x,t) = u(\xi) = u(x- \lambda t)\in H^1(\R),
\end{equation}
where $\lambda$ denotes the traveling speed and $\xi = x - \lambda t$ is the moving coordinate with speed $\lambda \in \R$. We are thus looking in the class of ``localized" solutions to the system
\begin{equation}\label{solieqn}
\left\{\begin{array}{l}
c\eta^{\prime\prime} + \eta - \lambda u + d \lambda u^{\prime\prime} + {1\over2} u^2 = 0,\\
au^{\prime\prime} + u -\lambda \eta + b \lambda \eta^{\prime\prime} + \eta u = 0,
\end{array}\right.
\end{equation}
where $\prime$ denotes the derivative with respect to $\xi$. The regularity of the solutions indicates the asymptotic behavior 
\begin{equation}\label{asymptotics}
\displaystyle \lim_{|x|\to \infty} (\eta, u) = (0,0).
\end{equation}

Note that when $b=d$, the system possesses a Hamiltonian structure with Hamiltonian
\begin{equation}\label{hamiltonian}
\mathcal{H}(\eta, u) = {1\over2} \int \LB -c\eta_x^2 - au_x^2 + \eta^2 + (1 + \eta) u^2 \RB dx.
\end{equation}
The solitary waves correspond to the critical points of the action functional $S_{\lambda} = \mathcal{H} - \lambda \mathcal{I}$, where 
\[
\mathcal{I}(\eta, v) = \int \LC \eta u + b \eta_x u_x \RC dx
\]
is called the impulse functional, and the Lagrange multiplier $\lambda$ gives the speed of the wave. 

From \eqref{hamiltonian} we see that the Hamiltonian $\mathcal H(\eta,u)$ is coercive in $H^1$ provided that $a, c < 0$. In this parameter regime, the existence of solitary waves can be inferred from the existence of minimizers to a constraint minimization problem \cite{CNS} under the assumptions that the surface tension is large ($\tau>1/3$) and $\|\eta\|_{H^2}$ is small.
%\begin{equation}\label{minprob}
%\inf_{\eta, u} \LCB \mathcal{H}(\eta, u): \ (\eta, u) \in B_r\times H^1,\  \mathcal{I} = 2\mu \RCB
%\end{equation}
%where $B_r$ is a ball of radius $r$ in $H^2$ and $\mu$ is a small positive number. 
Later in \cite{CNS1} another variational formulation was adapted in the same parameter regime to establish the existence of solitary waves for any $\tau\geq0$, but with a smallness restriction on the traveling speed $\lambda$. Using a Nehari manifold technique, the existence of ground state solutions (nontrivial solitary waves carrying minimum action energy $S_\lambda$) was established in \cite{BCL}. In the case of large surface tension $\tau > 1/3$, these ground states are shown to be depression waves which are symmetric and increasing from their unique troughs, consistent with the results in the context of two-dimensional full gravity-capillary water waves \cite{AH,Kirchgassner, Sachs}.

All the above analytical results are crucially based upon the Hamiltonian structure of the system, i.e., $b = d$. Our main goal is to extend the existence result to the cases when the parameters fall out of this regime. In particular, we will focus on pure gravity waves, corresponding to $\tau = 0$, and allow either (i) $b\ne d$, so that the Hamiltonian structure is no longer available; (ii) $a, c > 0$, so that the quadratic part of the Hamiltonian \eqref{hamiltonian} is not positive definite; or (iii) the wave speed $\lambda$ is large $|\lambda| > 1$, so that the action functional fails to be bounded from below. In all cases, the standard variational method seems hard to apply. 
%{\bf Here put the reference on nonexistence of subcritical waves as motivation for (ii)}

The main tool we are using is the bifurcation theory. For this to work we need to first choose a good parameter $s \in \R$ with which the problem \eqref{solieqn} can be formulated as an abstract one-parameter problem
\[
\mathscr{F}(U, s) = 0
\]
where $U := (u, \eta)$. The perturbative construction of solutions relies on a good understanding of the linearized operator $\mathscr{F}_U$ at some special solution $(U_0, s_0)$. It turns out that the translation invariance of the problem naturally generates a nontrivial kernel of the linearized operator $\mathscr{F}_U$ at any solution. With some appropriate choices of the ``base point solution" $(U_0, s_0)$, standard ODE techniques can be applied to ensure that the kernel is exactly one dimensional and hence can be removed by suitable choice of the function spaces, allowing us to invoke the Implicit Function Theorem to obtain a local curve of solutions.  

As is common for the solitary wave problem, continuing the local curve globally by standard global bifurcation techniques faces a serious obstruction due to the unboundedness of the domain. One classical approach is to approximate the solitary waves by periodic ones as the period tends to infinity. Such a method is used by Toland \cite{Toland} to treat \eqref{solieqn} with $(a,b,c,d) = (0, \frac13, -\frac13, \frac13)$. He first obtains a global bifurcation theory for the periodic problems, and then proves a uniform estimate. Together with an application of the Whyburn lemma, this leads to the convergence of the global sets of periodic solutions to a global connected set of solitary wave solutions as the period goes to infinity. 

We will adapt a recently developed analytic global implicit function theorem in \cite{CWW} for the global theory, cf. Theorem \ref{thm global bif}. As is pointed out in \cite{CWW}, the global curve may not be locally pre-compact, nor can one assume a priori that Fredholmness persists. Thus the loss of compactness emerges as an alternative. The ODE nature of the problem easily rules out the failure of Fredholmness. Therefore the theory will become useful in practice if we can rule out the loss of properness or classify how it manifests.

More specifically, we will consider global branches of solutions emanating from two base point solutions: the first one being the stationary solution (corresponding to $\lambda = 0$), and the second being the supercritical ($\lambda > 1$) waves to the classical Boussinesq system (corresponding to $(a,b,c,d) = (0,0,0,\frac13)$). We will also assign different parameters when studying these two types of waves. When bifurcating from the stationary waves, we use the wave speed $\lambda$ as the bifurcation parameter while fixing the $abcd$ system as in \eqref{parameters standing}, and obtain a continuous curve of solutions all the way into the regime where solutions are traveling with an $O(1)$ speed. For the other case we will fix an arbitrary supercritical speed $\lambda > 1$ and design a family of $abcd$ systems (as in \eqref{para curve 1}) that can accommodate solitary waves with such a speed $\lambda$. In both cases we prove a collection of qualitative properties of the solutions that are crucial for the final global result. In particular, using maximum principle arguments and the symmetry result for weakly coupled cooperative elliptic systems \cite{BS} we are able to obtain local uniqueness, local monotonicity, and nodal pattern of the solutions. The fact that we are always considering a {\it system} makes the maximum arguments more delicate, and possibly more restrictive; see Section \ref{subsec bif stationary}--\ref{subsec nodal} and Section \ref{subsec B local}--\ref{subsec nodal1}. 

Regarding the ruling-out/realization of the loss of compactness alternative in the global theory, as was studied in \cite{CWW1,CWW}, the established monotonicity property is strong enough to assert a ``compactness or front'' result stating that this possibility must manifest as a broadening phenomenon, leading to a {\it monotone front} type of solution at the end of the bifurcation curve. When the underlying system possesses a Hamiltonian structure, a so-called conjugate flow analysis can be carried out utilizing the conserved quantities to rule out the broadening alternative \cite{AW,CWW1,CWW,CWW2,Hogancamp,Sinambela}. Moreover, for some particular problems such a Hamiltonian structure may also allow one to obtain uniform bounds on solutions that can account for the realization of broadening \cite{Hogancamp}. In the cases we consider, however, the system is not Hamiltonian, and we do not have any obvious conserved quantities that can be of much use to control the solutions. Taking advantage of the monotonicity and together with delicate algebra we are able to prove the nonexistence of monotone front solutions, cf. Lemma \ref{lem no fronts} and Lemma \ref{lem no fronts 1}. Using this idea we can also prevent the blowup of solutions $(u, \eta)$ in the case of bifurcation from stationary waves, which leads to a sharp result ensuring the loss of ellipticity as the only remaining alternative cf. Theorem \ref{thm global slow}. For the other case of solutions bifurcating from the classical Boussinesq waves, we are able to winnow the alternatives down to the possibilities of either the loss of ellipticity or that the curve continues up to the appearance of an ``extreme wave'' that has a stagnation point, cf. Theorem \ref{thm global bif fast}.

%%%%%%%%%%%%%%%%%%%%%%%%%%%%%%%%%%%%%%%%%%%%%%%%%%%%%%%%%%%%%%%%%%%%%%%%%%%%%%%%%%%%%%%%%%%

\section{Bifurcation from stationary waves with $a = c < 0$} \label{sec stationary}
We start by constructing solutions near the stationary waves corresponding to $\lambda = 0$. To ensure ellipticity we will impose the sign condition $a, c < 0$.

\subsection{Stationary solutions}\label{subsec stationary}
Note that in the case when $\lambda=0$ the terms in system \eqref{solieqn} containing $b$ and $d$ disappear and becomes
\begin{equation}\label{standing eqn}
\left\{\begin{array}{ll} 
-c\eta ^{\prime \prime }=\eta +{\frac{u^{2}}{2}}, & \\
-au^{\prime \prime }=u(1+\eta ).\vspace{0.1in} & 
\end{array}\right.
\end{equation}
%with $a, c < 0$. Let $(x_{\mathrm{min}},x_{\mathrm{max}})$ denote the maximal interval of
%existence of a solution of~(\ref{ext}). Then the solution satisfies the
%`first integral' property 
%\begin{equation}
%-a(u^{\prime })^{2} - c(\eta^{\prime })^{2}=u^{2}(1+\eta )+\eta ^{2}+\mathrm{C}%
%~~\forall x\in (x_{\mathrm{min}},x_{\mathrm{max}}),  \label{eq1aa}
%\end{equation}%
%where C is a constant. Solitary wave solutions are non-constant solutions
%which exist for all $x,$ and 
%\begin{equation}
%\lim_{x\rightarrow \pm \infty }\left( u(x),u^{\prime }(x),\eta (x),\eta
%^{\prime }(x)\right) =(0,0,0,0).  \label{eq1a}
%\end{equation}%
%Substituting~(\ref{eq1a}) into~(\ref{eq1aa}) shows that $\mathrm{C} = 0,$
%hence  
By elliptic regularity we know that any solution of \eqref{standing eqn} is smooth and $\displaystyle \lim_{|x|\to \infty} (\eta^\prime, u^\prime) = (0,0)$. Hence solitary wave solutions satisfy the `first integral' property
\begin{equation}\label{eq1aav}
-a(u^{\prime })^{2}-c(\eta^{\prime })^{2}=u^{2}(1+\eta )+\eta ^{2}. 
\end{equation}

The existence theory for \eqref{standing eqn} has been has been studied systematically in \cite{CHT}. Here we collect some resulta that will be important for the later bifurcation argument. For the reader's convenience we provide their proofs in Appendix \ref{appendix}.
\begin{lem}\label{lem eta sign standing} 
Any solitary wave solution of \eqref{standing eqn} satisfies  
\begin{equation*}%  \label{ge3sadd}
\eta(x)<0 \quad \text{on } \R.
\end{equation*}
\end{lem}

\begin{proposition}[Existence and uniqueness of stationary waves \cite{CHT}]\label{prop standing}
When $a = c = -\beta^2 < 0$ we have 
\begin{enumerate}[label=\rm(\roman*)]
\item\label{standing soln 1} there is a solitary wave solution such that 
$u^-_0(x)<0$ on $\R$. Up to translation, 
\begin{equation}\label{soln standing nega}
u^-_0(x) = -{3\sqrt{2}\over2} \mathrm{sech}^2\left({x\over 2\beta} \right), \quad \mathrm{and} \quad \eta_0(x) = -{3\over2} \mathrm{sech}^2\left({x\over 2\beta} \right).
\end{equation}
This solution is unique among the class of functions $(u, \eta)$ where $u < \sqrt{2}$;
\item\label{standing soln 2} there is a solitary wave solution such that $%
u^+_0(x)>0$ on $\R$. Up to translation, 
\begin{equation}\label{soln standing posi}
u^+_0(x) = {3\sqrt{2}\over2} \mathrm{sech}^2\left({x\over 2\beta} \right), \quad \mathrm{and} \quad \eta_0(x) = -{3\over2} \mathrm{sech}^2\left({x\over 2\beta} \right).
\end{equation}
This solution is unique among the class of functions $(u, \eta)$ where $u > -\sqrt{2}$.
\end{enumerate}
\end{proposition}

Note that from elliptic regularity we easily see that $(u_0^\pm, \eta_0) \in H^\infty(\R) \times H^\infty(\R)$.

\subsection{Local theory}\label{subsec bif stationary}
Now we will construct a local curve of solutions nearby the stationary solution $(u_0,\eta_0)$. The parameters we are taking satisfy
\begin{equation}\label{parameters standing}
a =  c = -d = -\beta^2 < 0, \quad b = \frac13 + \beta^2.
\end{equation}
Obviously we see that $b \ne d$, and hence we are outside the Hamiltonian regime when the surface tension is small. For simplicity we will take $\tau = 0$ in the following discussion. To fit our argument in the framework of \cite{CWW}, we will consider the problem in H\"older spaces. 

Denote by $C_0(\R)$ the set of continuous functions vanishing at infinity and 
\[
C^{2+\alpha}_{\bdd}(\R) = \left\{ f\in C^2(\R): \ \| f \|_{C^{2+\alpha}} < +\infty  \right\}. 
\]
Define for $\alpha \in (0, 1)$ the following H\"older space
\begin{equation*}
\begin{split}
& \Xspace := \left( C^{2+\alpha}_{\bdd, \even}(\R) \cap C_0(\R) \right) \times \left( C^{2+\alpha}_{\bdd, \even}(\R) \cap C_0(\R) \right), \\ & \Yspace := \left( C^{\alpha}_{\bdd,\even}(\R) \cap C_0(\R) \right) \times \left( C^{\alpha}_{\bdd,\even}(\R) \cap C_0(\R) \right),
\end{split}
\end{equation*}
where the subscript `$\even$' denotes the restriction to even functions, and  The use of $C_0(\R)$ is to realize the asymptotic condition \eqref{asymptotics}.

Writing $U = (u, \eta)$, the system for solitary waves takes the following form
\begin{equation}\label{slow system}
\mathscr{F}(U, \lambda) := \begin{pmatrix} \displaystyle \mathcal L \left( u - \lambda\left( 1 + \frac{1}{3\beta^2} \right) \eta \right) + \left( \frac{\lambda}{3\beta^2} + u \right) \eta, \\\\
\displaystyle \mathcal L \left(\eta - \lambda u \right) + \frac12 u^2 \end{pmatrix} = 0,
%
%\left\{\begin{array}{l}
%\displaystyle \mathcal L \left( u - \lambda\left( 1 + \frac{1}{3\beta^2} \right) \eta \right) + \left( \frac{\lambda}{3\beta^2} + u \right) \eta = 0, \\\\
%\displaystyle \mathcal L \left(\eta - \lambda u \right) + \frac12 u^2 = 0,
%\end{array}\right.
\end{equation}
where 
\[
\mathscr F: \Xspace \to \Yspace,
\]
and $\mathcal L := 1 - \beta^2 \partial_x^2$ is an invertible operator from $C^{2+\alpha}_{\bdd, \even}(\R) \cap C_0(\R) \to C^{\alpha}_{\bdd, \even}(\R) \cap C_0(\R)$.
%Solving the second equation we get
%\[
%\eta = \lambda u - \frac12 \mathcal{L}^{-1} (u^2).
%\]
%Substituting this back to the first equation eliminates the unknow $\eta$ and leads to a single equation for $u$
%\begin{equation}\label{abs standing local}
%\begin{split}
%\mathscr{F}(u, \lambda) := & \  \left[ 1 - \lambda^2\left(1 + \frac{1}{3\beta^2} \right) \right] \mathcal{L} u + \frac{\lambda^2}{3\beta^2} u + \frac{\lambda}{2}\left( 3 + \frac{1}{3\beta^2} \right) u^2 \\
%& \ - \frac{\lambda}{6\beta^2} \mathcal{L}^{-1} (u^2) - \frac12 u \mathcal{L}^{-1}(u^2) = 0,
%\end{split}
%\end{equation}
%where 
%\[
%\mathscr F: H^2(\R) \times \R \to L^2(\R).
%\]
% 
%The discussion in Section \ref{subsec stationary} indicates that $\mathscr F(U^\pm_0, 0) = 0$ where $U^\pm_0 = (u^\pm_0, \eta_0)$. The linearized operator at the solution $(U^\pm_0,0)$ is 
%\begin{equation}\label{lin op stationary}
%\mathscr F_{U}(U^\pm_0, 0)[V] = \mathcal L v + \eta_0 v - 2\eta_0 \mathcal{L}^{-1} (\eta_0 v),
%\end{equation}
%where we have used the relation that
%\[
%\frac12 \mathcal{L}^{-1}(u_0^2) = \eta_0, \qquad \text{and} \qquad u_0^2 = 2 \eta_0^2.
%\]
%It is clear to see that $\mathscr F_{u}(u_0, 0)$ is self-adjoint on $L^2(\R)$. The following lemma states that the kernel of $\mathscr F_{u}(u_0, 0)$ is only generated by the translation symmetry.

The discussion in Section \ref{subsec stationary} indicates that $\mathscr F(U^\pm_0, 0) = 0$ where $U^\pm_0 := (u^\pm_0, \eta_0)$. The linearized operator at the solution $(U^\pm_0,0)$ is 
\begin{equation}\label{lin op stationary}
\mathscr F_{U}(U^\pm_0, 0)[V] = \mathcal{L} V + \begin{pmatrix} \eta_0 & u^\pm_0 \\ u^\pm_0 & 0 \end{pmatrix} V,
\end{equation}
where $V := (v, \zeta) \in \Xspace$. 
%It is clear to see that $\mathscr F_{U}(u_0, 0)$ is self-adjoint on $L^2(\R) \times L^2(\R)$. 
The following lemma states that the kernel of $\mathscr F_{U}(U^\pm_0, 0)$ is only generated by the translation symmetry.  
%%%%%%%%%%%%%%%%%%%%%%%%%%%%%%%%%%%%%%%%%
\begin{lem}\label{lem kernel stationary}
For any given $\beta >0$, $\mathscr F_{U}(U^\pm_0, 0): \Xspace \to \Yspace$ is injective.
%$\textup{ker} \mathscr F_{U}(U^\pm_0, 0) = \textup{span}\{ (U^\pm_0)^\prime \}$ where $(U^\pm_0) := \left(u^\pm_0, \eta_0 \right)$.
\end{lem}
%%%%%%%%%%%%%%%%%%%%%%%%%%%%%%%%%%%%%%%%%
\begin{proof}
Let $V = (v, \zeta) \in \textup{ker} \mathscr F_{U}(U^\pm_0, 0)$. Then $V$ satisfies
\begin{equation}\label{U eqn}
\mathcal{L} V + \begin{pmatrix} \eta_0 & u^\pm_0 \\ u^\pm_0 & 0 \end{pmatrix} V = 0.
\end{equation}
%The regularity condition implies that $V$ is bounded. 
Notice that the Green's function for $\mathcal{L}^{-1}$ is $G(x) = \frac{1}{2\beta} e^{-|x|/\beta}$. Therefore
\begin{equation*}
\begin{split}
V(x) & = - G(x) \ast \left[ \begin{pmatrix} \eta_0 & u^\pm_0 \\ u^\pm_0 & 0 \end{pmatrix} V \right](x) \\
& = - \int_{\R} G(x-y) \begin{pmatrix} \eta_0(y) & u^\pm_0(y) \\ u^\pm_0(y) & 0 \end{pmatrix} V(y) \,dy \\
& = -\frac{1}{G(x)} \int_{\R} \frac{G(x-y) G(y)}{G(x)} \frac{1}{G(y)}\begin{pmatrix} \eta_0(y) & u^\pm_0(y) \\ u^\pm_0(y) & 0 \end{pmatrix} V(y) \,dy.
\end{split}
\end{equation*}
Since 
\[
\LV \frac{G(x-y) G(y)}{G(x)} \RV \lesssim 1, \qquad \LV \frac{\eta_0(y)}{G(y)} \RV + \LV \frac{u^\pm_0(y)}{G(y)} \RV \lesssim 1,
\]
we conclude that $V$ decays exponentially
\be\label{U exp decay}
\LV G(x) V(x) \RV \lesssim 1. 
\ee

Expanding \eqref{U eqn} into a $4\times 4$ first order ODE system and checking the asymptotics we find that there are only two bounded solution branches, and they have the asymptotic behavior
\[
e^{-|x|/\beta} \qquad \text{and} \qquad |x| e^{-|x|/\beta} \qquad \text{as } \ |x| \to \infty.
\]
Together with \eqref{U exp decay} we know that the space of bounded solutions to \eqref{U eqn} is at most one-dimensional. Recalling from the translation invariance that 
\[
\mathscr F_{U}(U^\pm_0, 0) [(U^\pm_0)^\prime] = 0,
\]
it follows that $(U^\pm_0)^\prime$ is the only bounded solution to \eqref{U eqn}. Finally the parity condition yields the desired result. 
\end{proof}

The spectral property of $\mathscr F_{U}(U^\pm_0, 0)$ given by Lemma \ref{lem kernel stationary} allows a use of the Implicit Function Theorem. 
%when viewing $\mathcal F$ as a mapping $H^2_\even(\R) \times H_\even^2(\R) \times \R \to L_\even^2(\R) \times L_\even^2(\R)$, where the spaces
%\be\label{defn even space}
%H^2_\even(\R) := \{ f\in H^2(\R): \ f \text{ is even} \}, \qquad L^2_\even(\R) := \{ f\in L^2(\R): \ f \text{ is even} \}.
%\ee
%
Notice that for if $(u, \eta, \lambda)$ is a solution to \eqref{slow system} with $\lambda > 0$, then so is $(-u, \eta, -\lambda)$. In fact this corresponds to the same wave propagating in the opposite direction. Therefore in the following analysis we will only consider the case $\lambda > 0$.

%%%%%%%%%%%%%%%%%%%%%%%%%%%%%%%%%%%%%%%%%
\begin{theorem}[Nearly stationary waves]\label{thm slow wave existence}
For any $\beta\in \R$ there exists some positive $\lambda_0 > 0$ and a $C^0$ solution curve 
\[
\cm^{\textup{slow}}_{\textup{loc}} = \{ (u^\pm(\lambda), \eta(\lambda), \lambda): \ 0 \le \lambda < \lambda_0 \} \subset \Xspace \times \R
\]
to problem \eqref{slow system} with the property that 
\begin{align}
%(u^\pm(0), \eta(0)) = (u^\pm_0, \eta_0), \label{local starting pt} \\
%& u^+(-\lambda) = -u^-(\lambda), \qquad \eta(-\lambda) = \eta(\lambda), \label{local symmetry} \\
& u(\lambda) = u^+_0 + O(\lambda), \quad \eta(\lambda) = \eta_0 + O(\lambda) \quad \text{in }\ \Xspace, \label{local starting pt} \\
%& \lambda u(\lambda) > \eta(\lambda), \quad \text{and} \quad \eta(\lambda) < 0, \label{local sign fast eta} \\
& u(\lambda) > 0 \quad \text{and } \quad \eta(\lambda) < 0, \label{local sign fast u}
\end{align}
where $(u^+_0, \eta_0)$ is given in \eqref{soln standing posi}.
\end{theorem}
%%%%%%%%%%%%%%%%%%%%%%%%%%%%%%%%%%%%%%%%%
\begin{proof}
The proof of the existence and uniqueness of the solution curves and \eqref{local starting pt} follows from Lemma \ref{lem kernel stationary} and a direct application of the Implicit Function Theorem. 
%Moreover \eqref{local symmetry} follows from the symmetry
%\[
%\mathscr{F}(-u, \eta, -\lambda) = \mathscr{F}(u, \eta, \lambda),
%\]
%which in turn implies \eqref{local symmetry}.

%From \eqref{local symmetry} we also know that as $\lambda \to 0$
%\be\label{ift est}
%(u^\pm(\lambda), \eta(\lambda)) = (u^\pm_0, \eta_0) + O(\lambda) \quad \text{in }\ H^2_\even(\R) \times H^2_\even(\R).
%\ee

Applying the maximum principle to the second equation of \eqref{slow system} we see that $\lambda u \ge \eta$. 
From \eqref{slow system} we also have
\begin{equation}\label{eqn eta}
-\beta^2\LB 1 - \lambda^2 \LC 1 + \frac{1}{3\beta^2} \RC \RB \eta^{\prime\prime} +(1 - \lambda^2 - \lambda u) \eta + \frac12 u^2 = 0.
%\mathcal{L} \LB 1 - \lambda^2 \LC 1 + \frac{1}{3\beta^2} \RC \RB \eta + \LC \frac{\lambda^2}{3\beta^2} + \lambda u \RC\eta +  \frac12 u^2 = 0.
\end{equation}
From \eqref{local starting pt} we know that for $\lambda$ sufficiently small $1 - \lambda^2 - \lambda u > 0$. Therefore, from the maximum principle we conclude that $\eta \le 0$. If there is an $x_1$ such that $\eta(x_1) = 0 = \max \eta$, then we have $\eta'(x_0) = 0$. Substituting this into the above equation leads to $\eta^{\prime\prime}(x_0) = u(x_0) = 0$. Hence $(\eta - \lambda u)(x_0) = 0$. Since $\eta - \lambda u \le 0$, we see that $(\eta - \lambda u)(x_0) = \max (\eta - \lambda u)$, and thus $u^\prime(x_0) = 0$. The uniqueness of ODE then implies that $(\eta, u) \equiv 0$, a contradiction. Therefore we must have
\[
\eta < 0.
\]
%Applying this to the second equation of \eqref{slow system} again leads to $\lambda u > \eta$. This proves \eqref{local sign fast eta}.

Direct calculation yields the equation for $u$ as
\begin{equation}\label{eqn u}
\begin{split}
-\beta^2\LB 1 - \lambda^2 \LC 1 + \frac{1}{3\beta^2} \RC \RB u^{\prime\prime} & + \LB 1 - \lambda^2 \LC 1 + \frac{1}{3\beta^2} \RC + \frac{\lambda}{2} \LC 1 + \frac{1}{3\beta^2} \RC u \RB u \\
& + \LC \frac{\lambda}{3\beta^2} + u \RC \eta = 0.
\end{split}
\end{equation}
From \eqref{local starting pt} and \eqref{soln standing posi} we know that for any $\varepsilon > 0$ there exist $\lambda>0$ sufficient small and $R_0>0$ sufficiently large such that
\begin{equation}\label{cond R0 u eta}
\begin{split}
& \left\| u - u^+_0 \right\|_{C^2(\R)} + \left\| \eta - \eta_0 \right\|_{C^2(\R)} + \left\| u^+_0 \|_{C^0(|x| \ge R_0) } + \| \eta_0 \right\|_{C^0(|x| \ge R_0) } < \varepsilon, \\
& u > 0 \qquad \text{for} \quad |x| < R_0.
\end{split}
\end{equation} 

If $\inf u < -\frac{\lambda}{3\beta^2} < 0$, 
then from the above equation we know that there exists some $x_0 > R_0$ such that $u(x_0) = \inf u$. 
Continuity then yields the existence of $x_1$ with $x_1 > R_0$ and $u(x_1) = 0$ such that
\[
x_1 = \min \{x > 0: \ u(x) = 0 \}.
\]
Rewriting \eqref{eqn u} as
\begin{equation}\label{rewrite eqn u}
-\beta^2\LB 1 - \lambda^2 \LC 1 + \frac{1}{3\beta^2} \RC \RB u^{\prime\prime} + \LB 1 - \lambda^2 \LC 1 + \frac{1}{3\beta^2} \RC + \eta + \frac{\lambda}{2} \LC 1 + \frac{1}{3\beta^2} \RC u \RB u  + \frac{\lambda}{3\beta^2} \eta = 0,
\end{equation}
we see from \eqref{cond R0 u eta} that $|\eta| < 2\varepsilon$ on $[x_1, +\infty)$. Thus for $\lambda$ and $\varepsilon$ sufficiently small, applying the maximum principle on $[x_1, +\infty)$ yields that
\[
u \ge 0 \qquad \text{on }\ [x_1, +\infty),
\]
which is a contradiction. 

Therefore 
\[
\inf u \ge -\frac{\lambda}{3\beta^2}.
\]
Substituting this into \eqref{eqn u}, from the maximum principle we can infer that $u > 0$, which is \eqref{local sign fast u}.
%\[
%u^+ > 0 \qquad \text{when }\ \lambda > 0. 
%\]
%
%When $\lambda < 0$, and if $\inf u < 0$, then same as above, we can find the two points $x_0$ and $x_1$ such that
%\[
%x_0 > x_1 > R_0, \quad u(x_0) = \inf u < 0, \quad u(x_1) = 0 \quad \text{and} \quad x_1 = \min \{x > 0: \ u(x) = 0 \}.
%\]
%Applying again maximum principle to \eqref{rewrite eqn u} on $[x_1, +\infty)$ we obtain that 
%\[
%u < 0 \qquad \text{on }\ (x_1, +\infty).
%\]
%
%The situation on $\mathscr{C}^-$ can be treated in a similar way. We can consider $\lambda < 0$ but sufficiently close to zero. Using maximum principle again on the equation for $u$ leads to $\lambda u^+(\lambda) < 0$. Putting together we proved \eqref{local sign fast}.
\end{proof}

To investigate further the qualitative properties of the solutions, let us first recall the following result of \cite[Theorem 2]{BS} on weakly coupled elliptic systems.
%%%%%%%%%%%%%%%%%%%%%%%%%%%%%%%%%%%%%%%%%
\begin{theorem}[\cite{BS}]\label{thm BS}
If $(u,v)$ is a classical solution to the following elliptic system
\begin{equation*}
\left\{\begin{array}{rll}
\Delta u + g(u,v) = 0 \quad & \textup{in} \quad & \ \R^n,\\
\Delta v + f(u,v) = 0 \quad & \textup{in} \quad & \ \R^n,\\
u, \ v > 0 \quad & \textup{in} \quad & \ \R^n,\\
u(x), \ v(x) \to 0 \quad & \textup{as} \quad & |x| \to \infty,
\end{array}\right.
\end{equation*}
where $f, g \in C^1([0, \infty) \times [0, \infty), \R)$. Suppose further that
\begin{enumerate}[label=\rm(\roman*)]
\item \label{BS sign 1} $\displaystyle \frac{\partial g}{\partial v}, \frac{\partial f}{\partial u}$ are non-negative on $[0, \infty) \times [0, \infty)$; \textup{(quasi-monotonicity)}
\item \label{BS sign 2} $\displaystyle \frac{\partial g}{\partial u}(0, 0) < 0 \text{ and } \frac{\partial f}{\partial v}(0,0) < 0$;
\item \label{BS det} $\textup{det} A > 0$, where
\[ 
A := \begin{pmatrix} \displaystyle \frac{\partial g}{\partial u} & \displaystyle \frac{\partial g}{\partial v} \\\\ \displaystyle \frac{\partial f}{\partial u} & \displaystyle \frac{\partial f}{\partial v} \end{pmatrix}(0,0).
\]
\end{enumerate}
Then there exist points $x_0, x_1 \in \R^n$ such that $u(x) = u(|x - x_0|)$ and $v(x) = v(|x - x_1|)$. Moreover
\[
\frac{du}{dr_0} < 0 \qquad \text{and} \qquad \frac{dv}{dr_1} < 0,
\]
where $r_0 := |x - x_0|$ and $r_1 := |x - x_1|$.
\end{theorem}
%%%%%%%%%%%%%%%%%%%%%%%%%%%%%%%%%%%%%%%%%

From the above theorem we immediately obtain
%%%%%%%%%%%%%%%%%%%%%%%%%%%%%%%%%%%%%%%%%
\begin{lem}[Local monotonicity]\label{lem slow wave monotonicity}
Fix $\beta \in \R$. There exists $\lambda_0>0$ such that every solution $(u, \eta,\lambda) \in \mathscr{C}^{\textup{slow}}_{\textup{loc}}$ with $0 \le \lambda < \lambda_0$ is strictly monotone in that for $x > 0$,
\begin{equation}\label{monotonicity}
u' < 0 \qquad \text{and} \qquad \eta' > 0.
\end{equation}
\end{lem}
%%%%%%%%%%%%%%%%%%%%%%%%%%%%%%%%%%%%%%%%%
\begin{proof}
We see that $(u, \eta)$ satisfies equations \eqref{eqn u} and \eqref{eqn eta}. Setting $v := -\eta$ and putting it into the form as in Theorem \ref{thm BS} we find that 
\begin{align*}
g(u, v) & = -\frac{1}{\beta^2} u + \frac{\lambda}{3 \beta^4 B} v + \frac{1}{\beta^2 B} uv - \frac{\lambda}{2 \beta^2 B} \LC 1 + \frac{1}{3\beta^2} \RC u^2, \\
f(u, v) & = -\frac{1 - \lambda^2}{\beta^2 B} v + \frac{\lambda}{\beta^2 B} uv + \frac{1}{2\beta^2 B} u^2,
\end{align*}
where $B := \LB 1 - \lambda^2 \LC 1 + \frac{1}{3\beta^2} \RC \RB > 0$ for small $\lambda$. Direct computation shows that
\begin{equation*}
\begin{pmatrix} \displaystyle \frac{\partial g}{\partial u} & \displaystyle \frac{\partial g}{\partial v} \\\\ \displaystyle \frac{\partial f}{\partial u} & \displaystyle \frac{\partial f}{\partial v} \end{pmatrix}
= 
\begin{pmatrix} \displaystyle -\frac{1}{\beta^2} + \frac{1}{\beta^2 B} v - \frac{\lambda}{\beta^2 B} \LC 1 + \frac{1}{3\beta^2} \RC u & \displaystyle \frac{\lambda}{3 \beta^4 B} + \frac{1}{\beta^2 B} u \\\\ \displaystyle \frac{\lambda}{\beta^2 B} v + \frac{1}{\beta^2 B} u & \displaystyle -\frac{1 - \lambda^2}{\beta^2 B} + \frac{\lambda}{\beta^2 B} u \end{pmatrix}.
\end{equation*}
From Theorem \ref{thm slow wave existence} we know that $u, v > 0$ when $\lambda$ is small, which implies that \ref{BS sign 1}--\ref{BS det} of Theorem \ref{thm BS} are satisfied. Therefore \eqref{monotonicity} holds. 
\end{proof}

Another application of Theorem \ref{thm BS} to the local solution near the bifurcation point $(u_0^+, \eta_0, 0)$ is the following result on the local uniqueness of the solution curve $\mathscr{C}^{\textup{slow}}_{\textup{loc}}$. In particular this result shows that all classical solutions near $(u^+_0, \eta_0, 0)$ with $\lambda > 0$ must be even and monotone on the positive axis. 
%%%%%%%%%%%%%%%%%%%%%%%%%%%%%%%%%%%%%%%%
\begin{corollary}[Local uniqueness]\label{cor uniqueness}
Denote by $\mathcal{B}_r$ the ball of radius $r>0$ in $\left( C^{2}(\R) \cap C_0(\R) \right) \times \left( C^{2}(\R) \cap C_0(\R) \right) \times \R$ centered at $(u^+_0, \eta_0, 0)$. There exists $\varepsilon > 0$ such that for $\lambda > 0$,
\begin{equation}\label{unique general}
\mathscr{F}^{-1}(0) \cap \mathcal{B}_\varepsilon = \mathscr{C}^{\textup{slow}}_{\textup{loc}} \cap \mathcal{B}_\varepsilon.
\end{equation}
%From symmetry \eqref{local symmetry} we have for $\lambda < 0$,
%\[
%\mathscr{F}^{-1}(0) \cap \mathcal{B}_\varepsilon = \mathscr{C}^{\textup{slow},-}_{\textup{loc}} \cap \mathcal{B}_\varepsilon.
%\]
\end{corollary}
%%%%%%%%%%%%%%%%%%%%%%%%%%%%%%%%%%%%%%%%
\begin{proof}
Consider a solution $(u, \eta, \lambda)$ to equations \eqref{eqn eta}--\eqref{eqn u} with
\[
\left\| u - u_0^+ \right\|_{C^2(\R)} + \left\| \eta - \eta_0 \right\|_{C^2(\R)} + |\lambda| < \varepsilon.
\]
There exists an $R_0>0$ large enough such that
\be\label{cond R0 u eta unique}
\begin{split}
& \left\| u - u_0^+ \right\|_{C^2(\R)} + \left\| \eta - \eta_0 \right\|_{C^2(\R)} + |\lambda| + \left\| u_0^+ \|_{C^0(|x| \ge R_0) } + \| \eta_0 \right\|_{C^0(|x| \ge R_0) } < \varepsilon, \\
& u > 0, \quad \eta<0, \quad u' < 0, \quad \eta' > 0 \qquad \text{for} \quad |x| < R_0.
\end{split}
\ee
Hence if $\sup  \eta > 0$, then from continuity there exists $x_0 := \min\{ x > 0: \ \eta(x) = 0 \}$ such that $\eta(x_0) = 0$ and $x_0 > R_0$. From \eqref{cond R0 u eta unique} we see that 
\[
1 - \lambda^2 - \lambda u > 0 \qquad \text{on } \ [x_0, +\infty).
\]
Applying the maximum principle to \eqref{eqn eta} on $[x_0, +\infty)$ yields that $\eta \le 0$ on $[x_0, +\infty)$. Together with \eqref{cond R0 u eta unique}, this fact contradicts the assumption that $\sup \eta > 0$. Therefore we must have $\eta \le 0$.

In a similar way if $\inf u < 0$, then we may find $x_1 := \min\{ x > 0: \ u(x) = 0 \}$ such that $u(x_0) = 0$ and $x_0 > R_0$. 
The maximum principle applied to \eqref{rewrite eqn u} on $[x_1, +\infty)$ leads to the conclusion that $u \ge 0$, contradicting to the assumption that $\inf u < 0$. Thus $u \ge 0$. 

If there exists some $x_0 \ge 0$ such that $\eta(x_0) = 0$, then $\eta(x_0) = \sup  \eta$, and hence $\eta'(x_0) = 0$ and $\eta''(x_0) \le 0$. From \eqref{eqn eta} we find that $u(x_0) = 0$. This also means that $u(x_0) = \inf u$, and so $u'(x_0) = 0$. Uniqueness of the ODE then implies that $\eta = u \equiv 0$, which contradicts \eqref{cond R0 u eta unique}. The same argument applies to the situation if $u$ touches zero at some finite point. 

The above argument indicates that for any small $(u, \eta, \lambda) \in \mathscr{F}^{-1}(0) \cap \mathcal{B}_\varepsilon$,
\[
u > 0 \qquad \text{and} \qquad \eta < 0.
\]
Then for $\lambda > 0$ one may apply Theorem \ref{thm BS} to conclude that $u$ and $\eta$ are both even. Therefore the uniqueness of $\mathscr{C}^{\textup{slow}}_{\textup{loc}}$ within $\mathscr{F}^{-1}(0) \cap \LC  \Xspace\times \R^+ \RC$ gives \eqref{unique general}.
\end{proof}

\subsection{Nodal pattern}\label{subsec nodal}

Now for each fixed $\beta \in \R$ we introduce the set
\begin{equation}\label{set O}
\mathcal{O} := \left\{ (u, \eta, \lambda) \in \Xspace \times \R^+: \ 1 - \lambda^2 \LC 1 + \frac{1}{3\beta^2} \RC > 0 \right\}.
\end{equation}

The results of Theorem \ref{thm slow wave existence} and Lemma \ref{lem slow wave monotonicity} naturally suggest us to consider the following ``nodal properties"
\begin{subequations}\label{nodal}
\begin{alignat}{2}
u & > 0  &\qquad& \textrm{in }\ \R,  \label{nodal u}\\
\eta & < 0  && \textrm{in } \ \R, \label{nodal eta} \\
%u + \lambda \eta & > 0  &\qquad& \textrm{in }\ \R,  \label{nodal combo}\\
u' & < 0 && \textrm{in }\ \R^+, \label{nodal u prime}\\
\eta' & > 0 && \textrm{in }\ \R^+. \label{nodal eta prime}
\end{alignat}
\end{subequations}

%%%%%%%%%%%%%%%%%%%%%%%%%%%%%%%%%%%%%%%%
\begin{lem}[Open property]\label{lem open}
Let $(u_*, \eta_*, \lambda_*) \in \mathcal{O} \cap \mathscr{F}^{-1}(0)$ be given and suppose that it satisfies \eqref{nodal}. There exists $\varepsilon = \varepsilon(u_*, \eta_*, \lambda_*) > 0$ such that, if $(u, \eta, \lambda) \in \mathcal{O} \cap \mathscr{F}^{-1}(0)$ and
\begin{equation}\label{open ball}
\left\| u - u_* \right\|_{C^2(\R)} + \left\| \eta - \eta_* \right\|_{C^2(\R)} + |\lambda - \lambda_*| < \varepsilon,
\end{equation}
then $(u, \eta, \lambda)$ also satisfies \eqref{nodal}.
\end{lem}
%%%%%%%%%%%%%%%%%%%%%%%%%%%%%%%%%%%%%%%%
\begin{proof}
The proof of \eqref{nodal u} and \eqref{nodal eta} follows the same argument as in the proof of Corollary \ref{cor uniqueness} by replacing $(u_0^+, \eta_0, 0)$ with $(u_*, \eta_*, \lambda_*)$. 
%
%Recall that $(u, \eta, \lambda)$ solves equations \eqref{eqn eta}--\eqref{eqn u}. Given that $(u_*, \eta_*, \lambda_*)$ satisfies \eqref{nodal}, there exists $\varepsilon > 0$ sufficiently small and $R_0>0$ large enough such that
%\begin{equation}\label{cond R0 u eta open}
%\begin{split}
%& \left\| u - u_* \right\|_{H^2(\R)} + \left\| \eta - \eta_* \right\|_{H^2(\R)} + |\lambda - \lambda_*| + \left\| u_* \|_{L^\infty(|x| \ge R_0) } + \| \eta_* \right\|_{L^\infty(|x| \ge R_0) } < \varepsilon, \\
%& u > 0, \quad \eta<0, \quad u + \lambda \eta > 0, \quad u' < 0, \quad \eta' > 0 \qquad \text{for} \quad |x| < R_0.
%\end{split}
%\end{equation} 
%
%Hence if $\sup  \eta > 0$, then from continuity there exists $x_0 := \min\{ x > 0: \ \eta(x) = 0 \}$ such that $\eta(x_0) = 0$ and $x_0 > R_0$. From \eqref{cond R0 u eta open} we see that 
%\[
%1 - \lambda^2 - \lambda u > 0 \qquad \text{on } \ [x_0, +\infty).
%\]
%Applying maximum principle to \eqref{eqn eta} on $[x_0, +\infty)$ yields that $\eta \le 0$ on $[x_0, +\infty)$. Together with \eqref{cond R0 u eta open} it contradicts with $\inf_{x\in \R} \eta > 0$. Therefore we must have
%\[
%\eta < 0.
%\]
%
%In a similar way if $\inf u < 0$, then we may find $x_1 := \min\{ x > 0: \ u(x) = 0 \}$ such that $u(x_0) = 0$ and $x_0 > R_0$. Maximum principle applied to \eqref{rewrite eqn u} on $[x_1, +\infty)$ leads to $u \ge 0$, contradicting to the assumption that $\inf u < 0$. Thus
%\[
%u > 0.
%\]
The proof for \eqref{nodal u prime}--\eqref{nodal eta prime} then follows directly from the application of Theorem \ref{thm BS}.
\end{proof}

\begin{lem}[Closed property]\label{lem closed}
Let $\{(u_n, \eta_n, \lambda_n)\} \subset \mathcal{O} \cap \mathscr{F}^{-1}(0)$ be given and suppose that $(u_n, \eta_n, \lambda_n) \to (u, \eta, \lambda) \in \mathcal{O}\cap \mathscr{F}^{-1}(0)$ in $C^2(\R) \times C^2(\R) \times \R$. If each $(u_n, \eta_n, \lambda_n)$ satisfies \eqref{nodal}, then $(u, \eta, \lambda)$ also satisfies \eqref{nodal} unless $u = \eta \equiv 0$.
\end{lem}
%%%%%%%%%%%%%%%%%%%%%%%%%%%%%%%%%%%%%%%%
\begin{proof}
First we see that 
\begin{equation*}
\begin{split}
& u \ge 0, \qquad \eta \le 0, \qquad \lambda \ge 0, \qquad \text{and}\\
& u' \le 0, \qquad \eta' \ge 0 \qquad \text{in } \ \R^+.
\end{split}
\end{equation*}
If there exists $x_0$ such that $u(x_0) = 0$, then $u(x_0) = \inf u$, and hence $u'(x_0) = 0$. 
From the equation \eqref{eqn u} and maximum principle we see that $\eta(x_0) = 0$. Therefore $\eta(x_0) = \sup  \eta$. So $\eta'(x_0) = 0$. Thus from the uniqueness of ODE we know that $u = \eta \equiv 0$.
\end{proof}

%%%%%%%%%%%%%%%%%%%%%%%%%%%%%%%%%%%%%%%%
\begin{lem}[Nodal property]\label{lem noda}
If $\mathcal{K}$ is any connected subset of $\mathcal{O} \cap \mathscr{F}^{-1}(0)$ that contains $\mathscr{C}^{\textup{slow}}_{\textup{loc}}$, then every $(u, \eta, \lambda) \in \mathcal{K}$ exhibits \eqref{nodal}.
\end{lem}
%%%%%%%%%%%%%%%%%%%%%%%%%%%%%%%%%%%%%%%%
\begin{proof}
First note that each $(u(\lambda),\eta(\lambda),\lambda) \in \mathscr{C}^{\textup{slow}}_{\textup{loc}}$ satisfies \eqref{nodal}. Recall the definition of $\mathcal{B}_r$ in Corollary \ref{cor uniqueness}. Fix $0 < \lambda < \lambda_0$ and take $\varepsilon$ to be sufficiently small, the local uniqueness of $\mathscr{C}^{\textup{slow}}_{\textup{loc}}$ implies that
\[
\mathcal{K} \cap \mathcal{B}_\varepsilon = \mathscr{C}^{\textup{slow}}_{\textup{loc}} \cap \mathcal{B}_\varepsilon,
\]
and $\mathcal{K} \backslash \mathcal{B}_\varepsilon$ is the connected component containing $(u(\lambda),\eta(\lambda),\lambda)$. Applying Lemmas \ref{lem open} and \ref{lem closed} completes the proof. 
\end{proof}

\subsection{Monotone fronts}\label{subsec mf}
Next we define the concept of {\it monotone fronts}. 
\begin{definition}\label{def front}
For $\lambda > 0$ and $\lambda^2 \LC 1 + \frac{1}{3\beta^2} \RC < 1$, we say $(u, \eta, \lambda)$ is a monotone front solution of \eqref{slow system} if $(u, \eta) \in C^2_\bdd(\R) \times C^2_\bdd(\R)$, and
\begin{equation}\label{slow front}
\lim_{x\to +\infty}(u(x), \eta(x)) = (0, 0), \quad \text{and} \quad u > 0, \quad \eta < 0, \quad u' \le 0, \quad \eta' \ge 0 \quad \text{in }\ \R,
\end{equation}
where $C^2_\bdd(\R)$ is the set of $C^2$ functions with bounded norms. 
%Note that the first condition above ensures ellipticity of \eqref{slow system}. 
\end{definition}
%Applying a boot strap argument to a solution $(u, \eta) \in H^2(\R) \times H^2(\R)$ yields that in fact $(u, \eta) \in H^\infty(\R) \times H^\infty(\R)$.
%%%%%%%%%%%%%%%%%%%%%%%%%%%%%%%%%%%%%%%%
\begin{lem}[Nonexistence of monotone fronts]\label{lem no fronts}
There exists some $\beta_0 > 0$ such that if $\lambda > 0$ and $|\beta| < \beta_0$, then system \eqref{slow system} does not admit any monotone front solution in the sense of \eqref{slow front}.
\end{lem}
%%%%%%%%%%%%%%%%%%%%%%%%%%%%%%%%%%%%%%%%
\begin{proof}
Suppose $(u, \eta)$ is a monotone front solution to \eqref{slow system}. Then since $u, \eta$ are bounded and monotone, 
\[
(\bar u, \bar\eta) := \lim_{x\to -\infty} (u(x), \eta(x))
\]
exists, and $\bar u > 0$, $\bar\eta < 0$. Evaluating \eqref{eqn eta} at $-\infty$ leads to
\be\label{relation eta u}
\bar u < \frac{1 - \lambda^2}{\lambda}, \qquad \text{and} \quad \bar\eta = -\frac{\bar u^2}{2(1 - \lambda^2 - \lambda \bar u)}.
\ee
Substituting the above equation into \eqref{eqn u} and evaluating the equation at $-\infty$ yields
\begin{equation*}%\label{front root eqn}
-\LC 2 - B \RC {\bar u}^2 - \lambda B \bar u + 2(1 - \lambda^2) B = 0,
\end{equation*}
%where
%\[
%f(x) = -\LC 2 - B \RC x^2 - \lambda B x + 2(1 - \lambda^2) B,
%%f(x) = -\LB 1 + \lambda^2\LC 1 + \frac{1}{3\beta^2} \RC \RB \bar u^2 - \lambda \LB 1 - \lambda^2\LC 1 + \frac{1}{3\beta^2} \RC \RB \bar u + 2(1 - \lambda^2) \LB 1 - \lambda^2\LC 1 + \frac{1}{3\beta^2} \RC \RB = 0.
%\]
where $B = 1 - \lambda^2 \LC 1 + \frac{1}{3\beta^2} \RC \in (0, 1)$. Solving this quadratic equation  together with the constraint that $\bar u > 0$ yields
\begin{equation}\label{front root}
\bar u = \frac{\sqrt{\lambda^2 B^2 + 8B(2 - B)(1 - \lambda^2)} - \lambda B}{2(2 - B)}.
\end{equation}

On the other hand, multiplying \eqref{eqn eta} by $\eta'$ and multiplying \eqref{eqn u} by $u'$ and summing up, it follows that
\begin{equation}\label{int eqn}
\begin{split}
\LB -\frac{\beta^2 B}{2}\LC (u')^2 + (\eta')^2 \RC + \right. & \left. \frac{1 - \lambda^2}{2} \eta^2 + \frac{B}{2} u^2 + \frac{\lambda}{6}\LC 1 + \frac{1}{3\beta^2} \RC u^3 + \frac12 u^2 \eta \RB^\prime \\
&  - \lambda u \eta \eta' + \frac{\lambda}{3\beta^2} \eta u' = 0.
\end{split}
\end{equation}
We can rewrite the last two terms above as
\[
- \lambda u \eta \eta' + \frac{\lambda}{3\beta^2} \eta u' = \LC -\frac{\lambda}{2} u \eta^2 + \frac{\lambda}{3\beta^2} \eta u \RC^\prime + \frac{\lambda}{2} u^\prime \eta^2  - \frac{\lambda}{3\beta^2} \eta^\prime u.
\]
The definition of monotone front implies that $\frac{\lambda}{2} u^\prime \eta^2 - \frac{\lambda}{3\beta^2} \eta^\prime u \le 0$, and hence we have
\[
\frac{1 - \lambda^2}{2} \bar{\eta}^2 + \frac{B}{2} \bar{u}^2 + \frac{\lambda}{6}\LC 1 + \frac{1}{3\beta^2} \RC \bar{u}^3 + \frac{1}{2} \bar{u}^2 \bar{\eta} - \frac{\lambda}{2} \bar u \bar\eta^2 + \frac{\lambda}{3\beta^2} \bar u \bar\eta \le 0.
\]
Recalling \eqref{relation eta u} and the definition of $B$ the above inequality can be simplified to
\[
\bar u \LC \frac14 \bar\eta + \frac{B}{2} + \frac{1 - B}{6\lambda} \bar u \RC + \frac{\lambda}{3\beta^2} \bar\eta \le 0,
\]
which further leads to
\[
\bar u \LC \frac{B}{2} + \frac{1 - B}{6\lambda} \bar u \RC - \LC \frac14 \bar u + \frac{1 - B - \lambda^2}{\lambda} \RC \frac{\bar u^2}{2(1 - \lambda^2 - \lambda \bar u)} \le 0.
\]
Solving above yields
\[
\bar u \ge \frac{2\sqrt{4(1-B)^2(1-\lambda^2)^2 + 3(7 - 4B) B \lambda^2(1-\lambda^2)} - 4(1-B)(1 - \lambda^2)}{\lambda(7 - 4B)}.
\]
Combining this with \eqref{front root} and explicitly solving the resulting inequality leads to 
\begin{equation}\label{ineq h}
G(\lambda^2,t) \ge 0,
\end{equation}
where $t = 1 + \frac{1}{3\beta^2} > 1$ and
\begin{equation*}
G(z,t) := \LC -20 + \frac{13}{t} \RC z^3 + \LC -60 + \frac{33}{t} + 32 t \RC z^2 + \LC -39 + \frac{18}{t} + 32t \RC z - 9.
\end{equation*}
Recall from Definition \ref{def front} that we are only interested in the interval $z \in (0, 1/t^2)$. It is easy to see that 
\be\label{Gderiv at 0}
G_z(0, t), \ G_{zz}(0, t) > 0.
\ee
Looking at $G(z,t)$, we find that for $t > t_1$ sufficiently large, say $t_1 \approx 2.264$, we have
\[
G\LC \frac{1}{t^2}, t \RC < 0.
\]

For a fixed $t > 1$, solving a quartic inequality it follows that 
\begin{equation*}
G_{zz}\LC \frac{1}{t^2}, t \RC > 0 
\end{equation*}
when $t > t_2$ for some large enough $t_2$ (for example $t_2 \approx 1.68$). This together with \eqref{Gderiv at 0} and the fact that $G_{zz}(z, t)$ is linear in $z$ implies that for $t > t_2$, $G_{zz}(z, t) > 0$ for $0 < z < \frac{1}{t^2}$. Therefore we have
%\[
%G_{zz}(z, t) > 0 \quad \text{ for } \quad 0 < z < \frac{1}{t^2},
%\]
%which implies that 
\[
G_z(z, t) > G_z(0, t) > 0 \quad \text{ for } \quad 0 < z < \frac{1}{t^2}.
\]

So for $t > \max\{t_1, t_2\}$, corresponding to $\beta^2 < \beta_0^2$ with
\[
\beta_0^2 = \min\left\{ \frac{1}{{3t_1 - 1}}, \frac{1}{{3t_2 - 1}} \right\},
\]
it yields that 
\[
G(z, t) < 0  \quad \text{ for } \quad 0 < z < \frac{1}{t^2},
\]
which contradicts \eqref{ineq h}. This completes the proof of the lemma. 
\end{proof}
\begin{remark}
Taking $t_1 \approx 2.264$ and $t_2 \approx 1.68$, we may choose $\beta_0 \approx 0.5$.   
\end{remark}

%%%%%%%%%%%%%%%%%%%%%%%%%%%%%%%%%%%%%%%%%%
\subsection{Global continuation}\label{sec global}
Now that we have obtained the local bifurcation result, we will extend the local solution curves constructed in Section \ref{subsec bif stationary} to the non-perturbative regime using a global implicit function theorem developed in \cite{CWW}.

%In order to quote the result of \cite{CWW}, we define for $\alpha \in (0, 1)$ the following H\"older space
%\begin{equation*}
%\begin{split}
%& \Xspace := \left( C^{2+\alpha}_{\bdd, \even}(\R) \cap C_0(\R) \right) \times \left( C^{2+\alpha}_{\bdd, \even}(\R) \cap C_0(\R) \right), \\ & \Yspace := \left( C^{\alpha}_{\bdd,\even}(\R) \cap C_0(\R) \right) \times \left( C^{\alpha}_{\bdd,\even}(\R) \cap C_0(\R) \right),
%\end{split}
%\end{equation*}
%where the subscript `$\even$' denotes the restriction to even functions, and $C_0(\R)$ is the set of continuous functions vanishing at infinity. The use of $C_0(\R)$ is to realize the asymptotic condition \eqref{asymptotics}.
%
%Also we will modify $\mathcal O$ as
%\begin{equation*}
%\mathcal{O}_H := \left\{ (u, \eta, \lambda) \in \Xspace \times \R^+: \ 1 - \lambda^2 \LC 1 + \frac{1}{3\beta^2} \RC > 0 \right\}.
%\end{equation*}
%We will consider our problem in the above spaces instead. 
%
%Note that by elliptic regularity, the $H^2(\R) \times H^2(\R)$ solutions are indeed smooth, and hence $\mathcal{O} \cap \mathscr{F}^{-1}(0) = \mathcal{O}_H \cap \mathscr{F}^{-1}(0)$.

%Now we have all the needed properties to obtain the following global continuation. 
\begin{theorem}\label{thm global bif}
There exists a curve $\cm^{\textup{slow}}$ containing $\cm^{\textup{slow}}_{\textup{loc}}$, which admits a global $C^0$ parametrization
\[
\cm^{\textup{slow}} := \left\{ \LC u(s), \eta(s), \lambda(s) \RC: \ s\in (0, \infty) \right\} \subset \mathcal{O} \cap \mathscr{F}^{-1}(0)
\]
with $\lim_{s\searrow0}\LC u(s), \eta(s), \lambda(s) \RC = \LC u^+_0, \eta_0, 0 \RC$ and satisfies the following.
\begin{enumerate}[label=\rm(\alph*)]
  \item \label{well behaved} At each $s \in (0, \infty)$, the linearized operator $\F_{(u,\eta)}(u(s), \eta(s), \lambda(s)) \colon \Xspace \times \mathbb R^+ \to \Yspace$ is Fredholm index $0$.
  \item \label{alternatives} One of the following alternatives holds as $s \to \infty$.
    \begin{enumerate}[label=\rm(A\arabic*)]
    \item  \label{blowup alternative}
      \textup{(Blowup)}  The quantity 
      \begin{align}
        \label{blowup}
        N(s):= \|(u(s), \eta(s))\|_{\Xspace}+ \lambda(s) + \frac 1{\textup{dist}((u(s),\eta(s),\lambda(s)), \, \partial \mathcal{O})} \to \infty.
      \end{align}
    \item \label{loss of compactness alternative} \textup{(Loss of compactness)} There exists a sequence $s_n \to \infty$ with $\sup_n N(s_n) < \infty$, but $( u^+(s_n), \eta(s_n), \lambda(s_n) )$ has no convergent subsequence in $\Xspace \times \mathbb{R}^+$.  
      \item \label{loss of fredholmness alternative} \textup{(Loss of Fredholmness)}    There exists a sequence $s_n \to \infty$
        with $\sup_{n} N(s_n) < \infty$ and so that $(u(s_n), \eta(s_n), \lambda(s_n)) \to (u_*, \eta_*, \lambda_*)$ in $\Xspace \times \mathbb{R}^+$, however $\mathscr{F}_{(u,\eta)}(u_*, \eta_*, \lambda_*)$ is not Fredholm index $0$.  
   \item \label{loop alternative} \textup{(Closed loop)} There exists $T > 0$ such that $(u(s+T), \eta(s+T), \lambda(s+T)) = (u(s), \eta(s), \lambda(s))$ for all $s \in (0,\infty)$. 
        \end{enumerate}
          \item \label{global reparam} Near each point $(u(s_0),\eta(s_0), \lambda(s_0)) \in \cm^{\textup{slow}}$, we can locally reparameterize $\cm^{\textup{slow}}$ so that $s\mapsto (u(s),\eta(s), \lambda(s))$ is real analytic.
% \item \label{maximal part} The curve $\cm^{\textup{slow}}$ is maximal in the sense that, if $\mathscr{K} \subset \mathcal{O} \cap \mathscr{F}^{-1}(0)$ is a locally real-analytic curve containing $(0,0)$ and along which  $\F_u$ is Fredholm index $0$, then $\mathscr{K} \subset \cm$. 
  \end{enumerate}
\end{theorem}
\begin{proof}
The proof follows from \cite[Theorem B.1]{CWW} and \cite[Theorem 6.1]{CWW1}, since from Lemma \ref{lem kernel stationary} we know that $\F_{(u,\eta)}(u_0, \eta_0, \lambda_0) \colon \Xspace \times \mathbb R^+ \to \Yspace$ is an isomorphism. 
%Also note that the form of the blowup quantity in \eqref{blowup} does not involve $|\lambda(s)|$, since it is always bounded in $\mathcal O$.
\end{proof}

Given a $(u, \eta, \lambda) \in \Xspace \times \R$, direct computation yields that the linearized operator
\[
    \F_{(u, \eta)}(u,\eta, \lambda)=\begin{pmatrix}
    \mathcal{L}+\eta & -\lambda\left( 1 + \frac{1}{3\beta^2} \right)\mathcal{L}+\left( \frac{\lambda}{3\beta^2} + u \right) \\
    -\lambda\mathcal{L}+u & \mathcal{L}
    \end{pmatrix} : \ \Xspace \to \Yspace.
\]
Since $(u, \eta) \in \Xspace$, the limiting operator as $|x| \to \infty$ is
\[
\tilde \F_{(u, \eta)}(u,\eta, \lambda) := \begin{pmatrix}
    \mathcal{L} & -\lambda\left( 1 + \frac{1}{3\beta^2} \right)\mathcal{L} +  \frac{\lambda}{3\beta^2} \\
    -\lambda\mathcal{L} & \mathcal{L}
    \end{pmatrix}: \ \Xspace \to \Yspace.
\]

\begin{lem}\label{lem limiting inv}
For $(u, \eta, \lambda) \in \mathcal{O}$, the  limiting operator $\tilde \F_{(u, \eta)}(u,\eta, \lambda)$ is invertible.
\end{lem}
\begin{proof}
If $V = (v, \zeta) \in \Xspace$ such that $\tilde \F_{(u, \eta)}(u,\eta, \lambda)[V] = 0$, then a row elimination yields
\[
\left[ 1 - \lambda^2\left( 1 + \frac{1}{3\beta^2} \right) \right]\mathcal{L} \zeta + \frac{\lambda^2}{3\beta^2} \zeta = 0.
\]
Thus $\zeta = 0$, which also implies that $\mathcal{L} v = 0$, and hence $v = 0$. This shows that $\tilde \F_{(u, \eta)}(u,\eta, \lambda)$ is injective. 

Now for any $f = (f_1, f_2) \in \Yspace$, consider solving $\tilde \F_{(u, \eta)}(u,\eta, \lambda)[V] = f$ for $V\in \Xspace$. By a similar argument as before, we can perform a row elimination to solve for $\zeta$ from 
\[
\left[ 1 - \lambda^2\left( 1 + \frac{1}{3\beta^2} \right) \right]\mathcal{L} \zeta + \frac{\lambda^2}{3\beta^2} \zeta = \lambda f_1  + f_2,
\]
and then plug this back to the system to solve for $v$. This way we verify that $\tilde \F_{(u, \eta)}(u,\eta, \lambda)$ is also surjective. Therefore the conclusion follows.  
\end{proof}

With the help of Lemma \ref{lem limiting inv}, we may follow the argument in \cite{Wheeler,CWW1} to prove that $\F_{(u, \eta)}(u,\eta, \lambda): \Xspace \to \Yspace$ is locally proper. Finally we have
\begin{lem}\label{lem index}
For $(u, \eta, \lambda) \in \mathcal{O}$, the linearized operator $\F_{(u, \eta)}(u,\eta, \lambda)$ is Fredholm with index 0.
\end{lem}
\begin{proof}
The lemma can be proved by a homotopy argument. Consider the operator $L_t := t \tilde \F_{(u, \eta)}(u,\eta, \lambda) + (1- t)\left(\F_{(u, \eta)}(u,\eta, \lambda) - \tilde \F_{(u, \eta)}(u,\eta, \lambda) \right)$ for $t\in [0,1]$. Thus for any $t\in [0,1]$ the limiting operator of $L_t$ is $\tilde \F_{(u, \eta)}(u,\eta, \lambda)$. The previous argument proves that $L_t$ is locally proper, and thus Fredholm. Hence by continuity of the index we see that 
\[
\text{ind} \F_{(u, \eta)}(u,\eta, \lambda) = \text{ind} L_1 = \text{ind} \tilde \F_{(u, \eta)}(u,\eta, \lambda) = 0,
\]
which completes the proof. 
\end{proof}

Now we have all needed properties to further winnow down the alternatives in Theorem \ref{thm global bif}. By Lemma \ref{lem index}, for any $(u(s),\eta(s), \lambda(s)) \in \cm^{\textup{slow}}$, $\F_{(u, \eta)}(u(s),\eta(s), \lambda(s)) \colon \Xspace \to \Yspace$ is Fredholm index 0. Thus we know that \ref{loss of fredholmness alternative} does not occur.

The loop alternative \ref{loop alternative} can also be ruled out by the nodal property Lemma \ref{lem noda} combined with the uniqueness results Proposition \ref{prop standing} and Corollary \ref{cor uniqueness}.

As for \ref{loss of compactness alternative}, we may adapt \cite[Lemma 6.3]{CWW1} in our current setting to give the following
\begin{lem}[Compactness or front] \label{compactness or front lemma}
Suppose that $\{ (u_n,\eta_n, \lambda_n) \} \subset \F^{-1}(0) \cap \mathcal{O}$ satisfies 
\[
	\sup_{n\geq 1}{\left( \n{(u_n, \eta_n)}_{\Xspace} + \frac 1{\textup{dist}((u_n,\eta_n, \lambda_n), \, \partial \mathcal{O})} \right)} < \infty,
\]
and each $(u_n, \eta_n)$ is strictly monotone in that $\partial_x u_n < 0$, $\partial_x \eta_n > 0$ for $x > 0$. Then, either 
\begin{enumerate}[label=\rm(\roman*)]
\item \label{compactness alternative} \textup{(Compactness)} $\{ (u_n,\eta_n, \lambda_n)\}$ has a convergent subsequence in $\Xspace \times \mathbb{R}$; or
\item \label{front alternative} \textup{(Monotone front)} there exists a sequence of translations $x_n \to +\infty$ so that we can extract a convergent subsequence
\[ 
	(u_n, \eta_n)(\placeholder+x_n) \longrightarrow (u, \eta) \in C_\bdd^{2+\alpha}(\R) \quad \textrm{in } C_\loc^{2}(\R), \qquad \lambda_n \longrightarrow \lambda,
\]
with $(u, \eta, \lambda) \in \mathcal{O}$.  The limit is a \emph{monotone front solution} of \eqref{slow system} in the sense of Definition \ref{def front}.
\end{enumerate}
\end{lem}
\begin{proof}
Given the assumptions of the lemma, we know that up to a subsequence $\lambda_n \to \lambda$ with $1 - \lambda^2 \LC 1 + \frac{1}{3\beta^2} \RC > 0$. If $(u_n, \eta_n)$ is equi-decaying in the sense that for any $\varepsilon > 0$ there exists some $R > 0$ such that
\[
\sup_n \| (u_n, \eta_n) \|_{C^2((R,\infty))} < \varepsilon,
\]
then obviously $(u_n, \eta_n)$ has a convergent subsequence in $\Xspace$, and hence leads to \ref{compactness alternative}.

If $(u_n, \eta_n)$ is not equi-decaying, then there exists some $\varepsilon_0 > 0$ and a sequence $\{x_n\}$ with $x_n \to +\infty$ such that for all $n \ge 1$, 
\[
\sup_{0 \le i \le 2} \LV \partial_x^i (u_n, \eta_n)(x_n) \RV \ge \varepsilon_0.
\]
Set $(v_n, \zeta_n) := (u_n, \eta_n)(\placeholder + x_n)$. Since $(v_n, \zeta_n)$ is uniformly bounded in $\Xspace$, there is a subsequence, still denoted by the same labeling, $(v_n, \zeta_n) \to (u, \eta) \in \Xspace$ in $C^2_{\textup{loc}}(\R)$. Local convergence is enough to ensure that $(u^+, \eta)$ solves \eqref{slow system}. The monotonicity of $(u_n, \eta_n)$ confirms that 
\[
\partial_x u \le 0, \qquad \partial_x \eta \ge 0.
\]
By definition of $(v_n, \zeta_n)$ we see that
\[
\LV \partial_x^i (u, \eta)(0) \RV \ge \varepsilon_0 \qquad \text{for some } \ i \le 2.
\]
Thus $(u, \eta) \not\equiv (0, 0)$. The maximum principle then implies that $u > 0$ and $\eta < 0$.
\end{proof}

Putting all of the above together, we finally arrive at our main result of this section.
\begin{theorem}[Slow waves]\label{thm global slow}
For any $\beta$ with $|\beta| < \beta_0$ where $\beta_0$ is given in Lemma \ref{lem no fronts}, the global curve $\cm^{\textup{slow}}$ constructed in Theorem \ref{thm global bif} enjoys the following properties.
\begin{enumerate}[label=\rm(\alph*)]
\item \label{global monotone part} \textup{(Symmetry and monotonicity)} Each solution on $\cm^{\textup{slow}}$ is even and
\begin{equation}
  \begin{aligned}
    \eta(s) &< 0 & \quad & u(s) > 0 & \quad &  \textrm{on } \mathbb{R}, \\
    \partial_x \eta(s) & > 0 & \quad & \partial_x u(s) < 0 & \quad & \textrm{on } \mathbb{R^+}.
  \end{aligned} \label{monotonicity soln} 
\end{equation}
\item \label{loss of ellipticity part} \textup{(Loss of ellipticity)} Following $\cm^{\textup{slow}}$ to its extreme, the system loses ellipticity in that 
\begin{equation}\label{slow extreme}
\lim_{s \to \infty} \lambda(s) = \LC 1 + \frac{1}{3\beta^2} \RC^{-1/2}.
\end{equation}
\end{enumerate}
\end{theorem}
\begin{proof}
Note that property \ref{global monotone part} follows from the nodal properties Lemma \ref{lem noda}. From the previous discussion, at the extreme of the solution curve, \ref{loss of fredholmness alternative} and \ref{loop alternative} cannot occur. Lemma \ref{compactness or front lemma} together with Lemma \ref{lem no fronts} rules out \ref{loss of compactness alternative}. Therefore we are only left with {blowup alternative}. Since $\lambda$ is always bounded in $\mathcal O$, one can remove $\lambda(s)$ from the blowup quantity in \eqref{blowup}.

From the local uniqueness and the nodal properties we know that $\lim_{s \to \infty} \lambda(s) > 0$. So if \eqref{slow extreme} is false, then there exists a sequence $\{s_n\}$, $s_n \to \infty$ with the corresponding solutions $(u_n, \eta_n, \lambda_n) := (u, \eta, \lambda) (s_n) \in \mathcal{O} \cap \mathscr{F}^{-1}(0)$ such that
\begin{equation}\label{limit bu}
\lambda_n \to \lambda_* < \LC 1 + \frac{1}{3\beta^2} \RC^{-1/2}, \qquad \| (u_n, \eta_n) \|_{\Xspace} \to \infty.
\end{equation}
Moreover $\lambda_* > 0$. Since the system \eqref{slow system} is semi-linear, weakly coupled and has no first-order derivatives, elliptic regularity implies that $\| (u_n, \eta_n) \|_{C^0} \to \infty$. From \ref{global monotone part}, this is equivalent to
\[
u_n(0) - \eta_n(0) \to \infty.
\]
From the second equation in \eqref{slow system} and the fact that $(\eta_n - \lambda_n u_n)(0) = \min (\eta_n - \lambda_n u_n)$, it follows that
\[
\frac12 u_n^2(0) + (\eta_n - \lambda_n u_n)(0) = \beta^2 (\eta_n - \lambda_n u_n)^{\prime\prime}(0) \ge 0.
\]
From this it must hold that
\[
u_n(0) \to \infty.
\]

Similarly, evaluating \eqref{eqn u} at $x = 0$ and using that $u_n''(0) \le 0$ we find that
\[
\eta_n(0) \le - \frac{1 - \lambda_n^2\LC 1 + \frac{1}{3\beta^2} \RC + \frac{\lambda_n}{2} \LC 1 + \frac{1}{3\beta^2} \RC u_n(0)}{\frac{\lambda_n}{3\beta^2} + u_n(0)} u_n(0).
\]
For $n$ sufficiently large, from the above inequality, we have that
\be\label{bound eta 1}
\eta_n(0) \le -\frac{2\lambda_n}{5} \LC 1 + \frac{1}{3\beta^2} \RC u_n(0).
\ee

Recall \eqref{int eqn}. Integrating the equation over $(0, \infty)$ we find that 
\begin{equation*}
\frac{1 - \lambda_n^2}{2} \eta_n^2(0) + \frac{1}{2} u_n^2(0) \LB B + \frac{\lambda_n}{3} \LC 1 + \frac{1}{3\beta^2} \RC u_n(0) + \eta_n(0) \RB > 0.
\end{equation*}
From \eqref{bound eta 1} we obtain
\begin{equation}\label{eta u bd 1}
\frac{1 - \lambda_n^2}{2} \eta_n^2(0) + \frac{1}{2} u_n^2(0) \LB B - \frac{\lambda_n}{15} \LC 1 + \frac{1}{3\beta^2} \RC u_n(0) \RB > 0.
\end{equation}

Further using \eqref{slow system} we have that for any $\delta > 0$,
\begin{equation}\label{eqn combo}
\begin{split}
\mathcal{L} & \left\{ \delta \lambda_n\LC 1 + \frac{1}{3\beta^2} \RC \eta_n + \LB 1 - (1+\delta) \lambda_n^2\LC 1 + \frac{1}{3\beta^2} \RC  \RB u_n  \right\} \\
& \quad + \LC \frac{\lambda_n}{3\beta^2} + u_n \RC \eta_n + \frac12 (1 + \delta) \lambda_n \LC 1 + \frac{1}{3\beta^2} \RC u_n^2 = 0.
\end{split}
\end{equation}
From \eqref{limit bu} there exists $\delta_0 > 0$ such that for $n$ sufficiently large,
\[
1 - (1+\delta) \lambda_n^2\LC 1 + \frac{1}{3\beta^2} \RC > 0 \qquad \text{for all } \ 0 < \delta < \delta_0.
\]
From \eqref{eta u bd 1} we see that
\[
|\eta_n(0)| = O(|u_n(0)|^{3/2}) \qquad \text{as } \ n \to \infty. 
\]
Therefore for any $0 < \delta < \delta_0$ there exists some $n_0$ large enough such that for $n \ge n_0$
\[
\delta \lambda_n\LC 1 + \frac{1}{3\beta^2} \RC \eta_n(0) + \LB 1 - (1+\delta) \lambda_n^2\LC 1 + \frac{1}{3\beta^2} \RC  \RB u_n(0) < 0.
\] 
Denote by $x_n \in [0, \infty)$ the point where $\delta \lambda_n\LC 1 + \frac{1}{3\beta^2} \RC \eta_n + \LB 1 - (1+\delta) \lambda_n^2\LC 1 + \frac{1}{3\beta^2} \RC  \RB u_n$ achieves its minimum. Then it holds that $\eta_n(0) \le \eta_n(x_n) < 0$, and
\[
\delta \lambda_n\LC 1 + \frac{1}{3\beta^2} \RC \eta_n(x_n) \le \delta \lambda_n\LC 1 + \frac{1}{3\beta^2} \RC \eta_n(0) + \LB 1 - (1+\delta) \lambda_n^2\LC 1 + \frac{1}{3\beta^2} \RC  \RB \LB u_n(0) - u_n(x_n) \RB.
\]
From this we conclude that
\be\label{asympt eta}
|\eta_n(x_n)| = O(|u_n(0)|^{3/2}) \qquad \text{as } \ n \to \infty. 
\ee
Evaluating equation \eqref{eqn combo} at $x_n$ indicates that 
\[
\LC \frac{\lambda_n}{3\beta^2} + u_n(x_n) \RC \eta_n(x_n) + \frac12 (1 + \delta) \lambda_n \LC 1 + \frac{1}{3\beta^2} \RC u_n^2(x_n) > 0,
\]
which contradicts the asymptotics \eqref{asympt eta}.
\end{proof}

%%%%%%%%%%%%%%%%%%%%%%%%%%%%%%%%%%%%%%%%
\section{Bifurcation from classical Boussinesq supercritical waves} \label{sec supercritical}

In this section we focus on fast traveling solitary waves with wave speed $\lambda > 1$. Different from the previous section, here we will consider the wave speed as given, and restrict the four parameters $(a, b, c, d)$ on a one-parameter curve to perform the bifurcation. The base point of the bifurcation corresponds to the solution to the classical Boussinesq system which has $a = b = c = 0$ and $d = \frac13$ in \eqref{1dBoussinesq} (see, for example, \cite{Amick,Bous1,Peregrine,Schonbek}). As is discussed in \cite{Chen}, the solitary waves $(u_f, \eta_f)$ satisfy
\begin{equation}\label{fast base soln}
\left\{\begin{array}{ll} 
\displaystyle (u^\prime_f)^2 = \frac{1}{\lambda} \LC -u_f^3 + 3\lambda u_f^2 + 6 u_f + 6\lambda \log \LV \frac{\lambda - u_f}{\lambda} \RV \RC, & \medskip \\
\displaystyle \eta_f = \frac{u_f}{\lambda - u_f}. \vspace{0.1in} & 
\end{array}\right.
\end{equation}
From classical ODE techniques one obtains that for any $\lambda > 1$ there exists a unique solution $(u_f, \eta_f) \in \Xspace$ such that 
\begin{equation}\label{base soln prop}
\left\{\begin{split}
u_f, \eta_f & \text{ are both monotonically decreasing from their crests at } x = 0, \text{ and } \\
& \qquad \frac12\LC 3\lambda - \sqrt{\lambda^2 + 8} \RC < \max_{x\in \R} |u_f| < \lambda.
\end{split}\right.
\end{equation}

\subsection{Local solutions}\label{subsec B local}
Now for any fixed $k  > 0$ with $k  <  \lambda$, consider the parameter curve 
\begin{equation}\label{para curve 1}
a = c = k  s, \qquad b = s, \qquad d = \frac13 - (2k  + 1)s.
\end{equation}
Thus $b = d$ only when $2(k  + 1) s = \frac13$. So in particular $b \ne d $ when $s$ is small. Moreover we also allow $a, c$ to be negative.

Similar as before, in this parameter regime we can rewrite \eqref{solieqn} as 
\begin{equation}\label{fast system}
\mathscr{F}(U, s) := \begin{pmatrix} \displaystyle k  s u'' + \lambda s \eta'' + u - \lambda \eta + \eta u, \\\\
\displaystyle \LB \frac13 - (2k  + 1) s \RB \lambda u'' + k  s \eta'' - \lambda u + \eta + \frac12 u^2 \end{pmatrix} = 0,
\end{equation}
with $\mathscr F: \Xspace \times \R \to \Yspace$. 

The existence of solitary waves in the parameter regime \eqref{para curve 1} is stated as follows
%%%%%%%%%%%%%%%%%%%%%%%%%%%%%%%%%%%%%%%%%
\begin{theorem}[Fast waves near the Boussinesq solutions]\label{thm fast way existence}
For any $\lambda > 1$, let $k $ be such that $0 < k  < \lambda$. Suppose that the parameters of \eqref{solieqn} satisfy \eqref{para curve 1}. Then there exist some positive $\delta > 0$ and a unique $C^0$ solution curve 
\[
\cm^{\textup{fast}}_{\textup{loc}} = \{ (u_s, \eta_s, s): \ |s| < \delta \} \subset \Xspace \times \R
\]
to problem \eqref{fast system} with the property that 
\begin{align}
& (u_s, \eta_s) = (u_f, \eta_f) + O(s) \qquad \text{in} \quad \Xspace, \label{local est} \\
& u_s, \  \eta_s > 0 \qquad \text{for } \ s \ge 0, \label{local sign}
\end{align}
where $(u_f, \eta_f)$ is the unique solution to \eqref{fast base soln} satisfying \eqref{base soln prop}.
\end{theorem}
%%%%%%%%%%%%%%%%%%%%%%%%%%%%%%%%%%%%%%%%%
\begin{proof}
Denote $U_f := (u_f, \eta_f)$. Working with even functions, direct computation yields that
\begin{equation*}
\mathscr{F}_U(U_f, 0) = \begin{pmatrix} 1 + \eta_f & u_f - \lambda \\ 
\displaystyle \frac{\lambda}{3} \partial_x^2 + u_f - \lambda & 1 \end{pmatrix}: \Xspace \to \Yspace.
\end{equation*}

Suppose that $\mathscr{F}_U(U_f, 0)[V] = 0$ for some $V = (v, \zeta) \in \Xspace$. Writing out the equations we have
\begin{equation*}
\begin{split}
&(1+\eta_f)v+(u_f - \lambda)\zeta=0,    \\
&\frac{\lambda}{3}v''+(u_f - \lambda)v+\zeta=0.
\end{split}
\end{equation*}
We can then solve for $\zeta$ in the first equation to obtain a single ODE for $v$ 
\begin{equation}\label{reduced fast ode}
\frac{\lambda}{3}v''+ \LB \frac{\lambda}{(\lambda-u_f)^2}-(\lambda-u_f) \RB v=0.  
\end{equation}
Since $\lambda > 1$ and $u_f$ satisfies \eqref{base soln prop}, from classical ODE theory we know that there is only one bounded nontrivial solution to the above equation. On the other hand from the translation invariance of \eqref{fast system} we see that $u_f'$ solves \eqref{reduced fast ode}. From the fact that $u_f$ is even, it follows that $\textup{ker} \mathscr F_{U}(U_f, 0)$ is trivial in $\Xspace$.

The surjectivity of $\mathscr F_{U}(U_f, 0)$ can be easily verified since one can effectively solve $\zeta$ in terms of $v$ through an algebraic equation, and then solve an ODE for $v$. Therefore we further conclude that $\mathscr F_{U}(U_f, 0)$ is invertible. Thus the existence of the local solution curve and \eqref{local est} follows from the Implicit Function Theorem. 

Next let's turn to the sign property. For any $\varepsilon > 0$ we can find an $R_0 > 0$ such that 
\[
|u_f(x)|, \ |\eta_f(x)| < \varepsilon \qquad \text{for} \quad |x| > R_0.
\]
From \eqref{local est} we know that by choosing $s$ sufficiently small,
\begin{equation}\label{asympt sign}
\left\{\begin{array}{ll}
u_s(x), \ \eta_s(x) > 0 \quad & \text{for} \quad |x| \le R_0, \\
|u_s(x)|, \ |\eta_s(x)| < \varepsilon + O(s) \quad & \text{for} \quad |x| > R_0.
\end{array}\right.
\end{equation}
From \eqref{fast system} we have
\begin{equation}\label{fast u eqn}
\left( \LB \frac13 - (2k  + 1)s \RB \lambda^2 - k ^2 s  \right) u_s'' - \LB (k  + \lambda^2) - \frac{\lambda}{2} u_s \RB u_s + \LB \lambda + k  (\lambda - u_s) \RB \eta_s = 0.
\end{equation}
So if $\inf u_s = u_s(x_0) < 0$, then from \eqref{asympt sign} $|x_0| > R_0$. For small $s$, the maximum principle implies that $\eta_s(x_0) < 0$ and
\begin{equation*}
u_s(x_0) \ge \frac{\lambda + k  (\lambda - u_s(x_0))}{k  + \lambda^2 - \lambda u_s(x_0) / 2} \eta_s(x_0) > \eta_s(x_0),
\end{equation*}
since 
\[
0 < \frac{\lambda + k  (\lambda - u_s(x_0))}{k  + \lambda^2 - \lambda u_s(x_0) / 2}  < 1
\]
for sufficiently small $\varepsilon$ and $s$. 

Since $\eta_s(x_0) < 0$, we know that $\inf \eta_s = \eta_s(x_1) \le \eta_s(x_0) < 0$ for $|x_1| > R_0$. Looking at the equation for $\eta_s$
\begin{align}
%\begin{split}
\left( \LB \frac13 - (2k  + 1)s \RB \lambda^2 s - k ^2 s^2  \right) \eta_s'' & - \LC \LB \frac13 - (2k  + 1)s \RB \lambda^2 + k  s - \LB \frac13 - (2k  + 1)s \RB \lambda u_s \RC \eta_s \nonumber\\
& + \LC \LB \frac13 - (2k  + 1)s \RB \lambda+\lambda ks - \frac{k  s}{2} u_s \RC u_s = 0, \label{fast eta eqn}
%\end{split}
\end{align}
it follows that for $s> 0$ small, at $x_1$ we have $u_s(x_1) < 0$, and 
\[
\eta_s(x_1) \ge \frac{\LB \frac13 - (2k  + 1)s \RB \lambda +\lambda ks- \frac{k  s}{2} u_s(x_1)}{\LB \frac13 - (2k  + 1)s \RB \lambda^2 + k  s - \LB \frac13 - (2k  + 1)s \RB \lambda u_s(x_1)} u_s(x_1) > u_s(x_1).
\]
The last estimate holds because for $\varepsilon, s$ sufficiently small the fraction can be made between 0 and 1. However this would lead to a contradiction since
\[
\inf u_s = u_s(x_0) \ge \eta_s(x_0) \ge \inf \eta_s = \eta_s(x_1) > u_s(x_1).
\]
Therefore we have proved that for $s > 0$ sufficiently small, $u_s \ge 0$. A similar argument yields that $\eta_s \ge 0$ as well. 

If there is a point $x_*$ where $u_s(x_*) = 0$, then the above argument shows that $\eta_s(x_*) = 0$, and hence $x_*$ is a minimum point for $u_s$ and $\eta_s$, indicating that $u_s'(x_*) = \eta_s'(x_*) = 0$. Thus from uniqueness of ODE it must hold that $u_s = \eta_s \equiv 0$, which is a contradiction. This proves \eqref{local sign}. 
\end{proof}
Similar to Section \ref{sec stationary}, we have the following result establishing the local monotonicity and local uniqueness.
\begin{corollary}[Local monotonicity and local uniqueness]\label{mono and unique}
Denote by $\mathcal{B}_r$ the ball of radius $r>0$ in $\left( C^{2}(\R) \cap C_0(\R) \right) \times \left( C^{2}(\R) \cap C_0(\R) \right) \times \R$ centered at $(u_f,\eta_f,0)$. There exists $\varepsilon>0$ such that for $s>0$,
\begin{equation}
\mathcal{F}^{-1}(0)\cap \mathcal{B}_{\varepsilon} = \cm^{\textup{fast}}_{\textup{loc}}\cap\mathcal{B}_{\varepsilon}
\end{equation}
In addition, every solution $(u,\eta,s)\in \mathcal{F}^{-1}(0)\cap \mathcal{B}_{\varepsilon}$ is strictly monotone in that for $x>0$,
\begin{equation}
u' < 0 \qquad \text{and} \qquad \eta' < 0.
\end{equation}
\end{corollary}
\begin{proof}
Similarly as the proof of the sign property in Theorem \ref{thm fast way existence}, for $(u,\eta,s)\in \mathcal{F}^{-1}(0)\cap \mathcal{B}_{\varepsilon}$, we have $u,\eta>0$. Thus it suffices to check conditions \ref{BS sign 1}--\ref{BS det} of Theorem \ref{thm BS}.  Writing \eqref{fast u eqn} and \eqref{fast eta eqn} as 
\[
\left\{\begin{array}{l}
\displaystyle \left( \LB \frac13 - (2k  + 1)s \RB \lambda^2 s - k ^2 s^2  \right) u'' + g(u, \eta) = 0, \\\\
\displaystyle \left( \LB \frac13 - (2k  + 1)s \RB \lambda^2 s - k ^2 s^2  \right) \eta'' + f(u, \eta) = 0,
\end{array}\right.
\]
direct computation yields that:
\begin{equation}\label{comp BS}
\begin{split}
&\frac{\partial g}{\partial\eta}=\lambda+\lambda k-ku,\ \ \ \ \ \ \frac{\partial f}{\partial u}= \LB \frac13 - (2k  + 1)s \RB (\lambda + \eta) + ks (\lambda - u), \\
&\frac{\partial g}{\partial u}(0,0)=-(\lambda^2+k),\ \ \frac{\partial f}{\partial \eta}(0,0)=-\LB \frac13 - (2k  + 1)s \RB \lambda^2-ks.   
\end{split}
\end{equation}
When $\varepsilon$ is chosen sufficiently small, conditions \ref{BS sign 1}--\ref{BS det} of Theorem \ref{thm BS} are satisfied.
\end{proof}
\begin{remark}
In the proof above we used a relaxed version of condition \ref{BS sign 1} which only requires that $\displaystyle \frac{\partial g}{\partial v}(u_{\alpha},v), \frac{\partial f}{\partial u}(u,v_{\alpha})$ are non-negative for $(u,v)\in [0, \infty) \times [0, \infty)$ where $(u_{\alpha},v_{\alpha})$, where $(u_{\alpha},v_{\alpha})$ are reflection of the solution of the elliptic system with repect to the line $x = \alpha$.
\end{remark}
% MC: this remark needs to be rephrased.

\subsection{Nodal pattern and monotone fronts}\label{subsec nodal1}
Now for each fixed $\lambda,k\in\mathbb{R}^+$ we introduce the set 
\begin{equation}\label{set O}
\mathcal{O} := \left\{ (u, \eta, s) \in \Xspace \times \R^+: s\in\Gamma_1, u\in \Gamma_2  \right\}.
\end{equation}
where 
\begin{align*}
\Gamma_1 := \left\{s\in \mathbb{R}^+:\ \frac{\lambda^2}{3} - \LB (2k + 1)\lambda^2 + k^2 \RB s>0 \right\}, \quad 
\Gamma_2 := \left\{u:\|u\|_{C^0(\mathbb{R})} < \lambda \right\}.
\end{align*}
The intuition for the choice of $\mathcal{O}$ is that $\Gamma_1$ is needed for the ellipticity, and $\Gamma_2$ provides a sufficient condition to ensure conditions \ref{BS sign 1}--\ref{BS det} in Theorem \ref{thm BS}, in particular the condition \ref{BS sign 1} for (strict) quasi-monotonicity. Indeed from \eqref{comp BS} we see that $\partial_\eta g, \partial_u f > 0$ when
\[
u < \lambda + \frac{\lambda}{k} \qquad \text{and} \qquad u < \lambda + \frac{\frac13 - (2k+1)s}{ks} (\lambda + \eta).
\]
Moreover, this constraint also allows one to deduce from \eqref{fast eta eqn} an upper bound for $\eta$
\be\label{upper bd eta}
\eta_s(0) \le \frac{\LB \frac13 - (2k+1)s \RB \lambda + \LC \lambda - \frac{u_s(0)}{2} \RC ks}{\LB \frac13 - (2k+1)s \RB \lambda \big( \lambda - u_s(0) \big) + ks} u_s(0).
\ee
Constraint $\Gamma_2$ can also be understood as a ``no stagnation" condition and indicates that the particles travel behind the wave. 
%When $\lambda \in \Gamma_1$ and $\eta \ge 0$ it follows that
%\[
%\lambda + \frac{\frac13 - (2k+1)s}{ks} (\lambda + \eta) > \lambda + \frac{\lambda^3}{k^3} > \lambda + \frac{\lambda}{k}.
%\]

From Theorem \ref{thm fast way existence} and Corollary \ref{mono and unique} we are led to consider the following nodal property:
\begin{subequations}\label{nodal1}
\begin{alignat}{2}
u & > 0, \quad \eta > 0 &\qquad& \textrm{in }\ \R,  \label{nodal1 u, eta}\\
%\eta & > 0  && \textrm{in } \ \R, \label{nodal1 eta} \\
u' & < 0, \quad \eta' < 0 && \textrm{in }\ \R^+, \label{nodal1 u eta prime}
%\eta' & < 0 && \textrm{in }\ \R^+. \label{nodal1 eta prime}
\end{alignat}
\end{subequations}

Similarly to the previous section, we can prove that the above nodal property persists on the solution curve. The proof follows along the same line as the one in Lemma \ref{lem noda}, and hence we omit it. 
%%%%%%%%%%%%%%%%%%%%%%%%%%%%%%%%%%%%%%%%
\begin{lem}[Nodal property]\label{lem noda1}
If $\mathcal{K}$ is any connected subset of $\mathcal{O} \cap \mathscr{F}^{-1}(0)$ that contains $\mathscr{C}^{\textup{fast}}_{\textup{loc}}$, then every $(u, \eta, \lambda) \in \mathcal{K}$ exhibits \eqref{nodal1}.
\end{lem}
%We would drop the proof here since it is exactly the same argument as above.

%\subsection{Monotonce fronts}\label{subsec mf1} 
The next step regards the nonexistence of monotone fronts, which will provide useful information for the global theory. As in Section \ref{subsec mf}, we define the concept of monotone fronts as follows.
\begin{definition}\label{def front1}
Let $s\in \Gamma_1, u\in \Gamma_2$. we say $(u, \eta, \lambda)$ is a monotone front solution of \eqref{fast system} if $(u, \eta) \in C^2_\bdd(\R) \times C^2_\bdd(\R)$, and
\begin{equation}\label{fast front}
\lim_{x\to +\infty}(u(x), \eta(x)) = (0, 0), \quad \text{and} \quad u > 0, \quad \eta > 0, \quad u' \le 0, \quad \eta' \leq 0 \quad \text{in }\ \R.
\end{equation}
%where $C^2_\bdd(\R)$ is the set of $C^2$ functions with bounded norms. 
%Note that the first condition above ensures ellipticity of \eqref{slow system}. 
\end{definition}
\begin{lem}[Nonexistence of monotone fronts]\label{lem no fronts 1}
If $s\in \Gamma_1, u\in \Gamma_2$ then system \eqref{fast system} does not admit any monotone front solution in the sense of \eqref{fast front}.
\end{lem}
\begin{proof}
The proof is similar to Lemma \ref{lem no fronts} but the algebra is simpler. Suppose $(u, \eta)$ is a monotone front solution to \eqref{fast system}. Let 
\[
(\bar u, \bar\eta) := \lim_{x\to -\infty} (u(x), \eta(x)).
\]
So $\bar u,\bar\eta >0$. Evaluating \eqref{fast system} at $x \to -\infty$ implies that 
\begin{equation}\label{front root 1}
\bar u = \frac{3\lambda-\sqrt{\lambda^2+8}}{2}
\end{equation}
and 
\begin{equation}\label{relation eta u 1}
    \bar\eta = \frac{\bar u}{\lambda-\bar u}.
\end{equation}
Multiplying the first equation in \eqref{fast system} by $\eta'$, the second equation by $u'$, and then summing them gives
\begin{equation}\label{int eqn 1}
\begin{split}
&\LB \frac{\LC \frac{1}{3}-(2k+1)s\RC\lambda}{2}(u')^2+\frac{\lambda s}{2}(\eta')^2+ksu'\eta'+u\eta-\frac{\lambda}{2}u^2-\frac{\lambda}{2}\eta^2+\frac{1}{6}u^3
\RB'\\
& +u\eta\eta'=0
\end{split}
\end{equation}
Write $u\eta\eta'= \LC\frac{1}{2}u\eta^2\RC'-\frac{1}{2}u'\eta^2$, by definition of monotone front, we have $u'\eta^2\leq 0$ and thus
\begin{equation}\label{fast int}
    \bar u\bar\eta-\frac{\lambda}{2}\bar u^2-\frac{\lambda}{2}\bar\eta^2+\frac{1}{6}\bar u^3+\frac{1}{2}\bar u\bar\eta^2\geq 0
\end{equation}
From \eqref{relation eta u 1}, this can be reduced to 
\begin{equation*}
\LC \lambda-\bar u\RC\LC \frac{\bar u}{3}-\lambda\RC+1\geq 0    
\end{equation*}
and from \eqref{front root 1} we finally have 
\begin{equation*}
    \frac{1}{12}\LC-\lambda+\sqrt{\lambda^2+8}\RC\LC-3\lambda-\sqrt{\lambda^2+8}\RC+1\geq 0
\end{equation*}
the above inequality holds only when $\lambda\leq 1$, which contradicts the fact that $\lambda>1$.
\end{proof}
\subsection{Global continuation}\label{sec global 1}
As in Section \ref{sec global}, with Lemma \ref{lem noda1} and \ref{lem no fronts 1}, we obtain the following global solution curve:
\begin{theorem}\label{thm global bif fast}
There exists a curve $\cm^{\textup{fast}}$ containing $\cm^{\textup{fast}}_{\textup{loc}}$, which admits a global $C^0$ parametrization
\[
\cm^{\textup{fast}} := \left\{ \LC u(t), \eta(t), s(t) \RC: \ t\in (0, \infty) \right\} \subset \mathcal{O} \cap \mathscr{F}^{-1}(0)
\]
with $\lim_{t\searrow0}\LC u(t), \eta(t), s(t) \RC = \LC u_f, \eta_f, 0 \RC$ and satisfies the following property:
\begin{enumerate}[label=\rm(\alph*)]
\item \label{global monotone part 1} \textup{(Symmetry and monotonicity)} Each solution on $\cm^{\textup{fast}}$ is even and
\begin{equation*}
  \begin{aligned}
    u(t) &> 0 & \quad & \eta(t) > 0 & \quad &  \textrm{on } \mathbb{R}, \\
    \partial_x u(t) & < 0 & \quad & \partial_x \eta(t) < 0 & \quad & \textrm{on } \mathbb{R^+}.
  \end{aligned} \label{monotonicity soln} 
\end{equation*}
\item \label{loss of ellipticity part 1} \textup{(Loss of ellipticity or stagnation limit)} Following $\cm^{\textup{fast}}$ to its extreme, either the system loses ellipticity in that 
\begin{equation}\label{fast extreme 1}
\lim_{t \to \infty} s(t) = \frac{\lambda^2}{3\LC (2k+1)\lambda^2+k^2\RC},
\end{equation}
or we encounter waves that are arbitrarily close to having a stagnation point 
\begin{equation} \label{fast extreme 2}
 \lim_{t \to \infty} \inf_{x\in \R} \LC \lambda - u(t) \RC = 0.
\end{equation}
\end{enumerate}
\end{theorem}
\begin{proof}
We will only focus on proving \ref{loss of ellipticity part 1}. Since $\lim_{t\to\infty}s(t) > 0$, so if \eqref{fast extreme 1} is false, we can find a sequence $\{t_n\} \to \infty$ with the corresponding solutions denoted by $(u_n, \eta_n, s_n) \in \mathcal{O} \cap \mathscr{F}^{-1}(0)$ such that as $n \to \infty$,
\[
s_n \to s_* < \frac{\lambda^2}{3\LC (2k+1)\lambda^2+k^2\RC}, \qquad \text{either }\ \|(u_n, \eta_n)\|_{\Xspace} \to \infty \text{ or }\ \|u_n\|_{C^0} \to \lambda.
\]
Recall from \eqref{upper bd eta} the upper bound for $\eta_n$
\[
\bar\eta_n \le \frac{\bar s_n \lambda + \LC \lambda - \frac{\bar u_n}{2} \RC ks_n}{\bar s_n \lambda (\lambda - \bar u_n) + ks_n} \bar u_n \le \frac{2 \bar s_n \lambda + \lambda k s_n}{2k s_n} \lambda,
\]
where $\bar s_n := \frac13 - (2k+1)s_n$. Elliptic regularity then asserts that the latter alternative can be replaced by
\[
\|u_n\|_{C^0} \to \lambda,
\]
proving \eqref{fast extreme 2}.
%
%
%For each $(u_n, \eta_n, s_n)$, evaluating equation \eqref{fast system} at $x=0$ and denoting $(\bar u_n, \bar\eta_n) := (u_n(0), \eta(0))$ we obtain
%\[
%\bar u_n - (\lambda - \bar u_n) \bar\eta_n \ge 0, \qquad -\lambda \bar u_n + \bar\eta_n + \frac12 \bar u_n^2 \ge 0,
%\]
%leading to
%\be\label{bd eta 1}
%\LC \lambda - \frac{\bar u_n}{2} \RC \bar u_n \le \bar\eta_n \le \frac{\bar u_n}{\lambda - \bar u_n}.
%\ee
%Solving the above yields
%\[
%\LC \lambda - \frac{\bar u_n}{2} \RC (\lambda - \bar u_n) \le 1.
%\]
%Further denoting $\bar s_n := \frac13 - (2k+1)s_n$. From \eqref{upper bd eta} we have
%\[
%\bar\eta_n \le \frac{\bar s_n \lambda + \LC \lambda - \frac{\bar u_n}{2} \RC ks}{\bar s_n \lambda (\lambda - \bar u_n) + ks} \bar u_n.
%\]
%Together with \eqref{bd eta 1} we have
%\[
%\LC \lambda - \frac{\bar u_n}{2} \RC \bar u_n \le \frac{\bar s_n \lambda + \LC \lambda - \frac{\bar u_n}{2} \RC ks}{\bar s_n \lambda (\lambda - \bar u_n) + ks} \bar u_n,
%\]
%which leads to 
%\[
%\LC \lambda - \frac{\bar u_n}{2} \RC (\lambda - \bar u_n) \le 1.
%\]
\end{proof}

\section*{Acknowledgements}
R. M. Chen and J. Jin are supported in part by the NSF grants DMS-1613375 and DMS-1907584.

\appendix

\section{Proofs from Section \ref{subsec stationary}}\label{appendix}

In this appendix we provide the proofs of the properties for the stationary wave problem \eqref{standing eqn} stated in Section \ref{subsec stationary}.

\begin{proof}[Proof of Lemma \ref{lem eta sign standing}]
If there is an $x_{1}$ such that $\eta (x_{1})\geq 0.$ This property,
and the fact that $\eta (x)\rightarrow 0$ as $x\rightarrow \pm \infty ,$
imply that $\eta $ has a non-negative maximum on $\R.$ Suppose that this maximum is at $x_{0}.$ 
 Then from \eqref{standing eqn}, either $\eta \left( x_{0}\right) =u\left(
x_{0}\right) =0$ or $\eta ^{\prime \prime }\left( x_{0}\right) >0.$ The
latter being impossible at a maximum, we conclude the former two equalities.
Also, $\eta ^{\prime }\left( x_{0}\right) =0,$ and from (\ref{eq1aav}), $%
u^{\prime }\left( x_{0}\right) =0.$ But then, since $\eta ,$ $\eta ^{\prime
},$ $u,$ and $u^{\prime }$ all vanish at $x_{0},$ the uniqueness theorem for
ODE's implies that $\eta \left( x\right) =u\left( x\right) =0$ for all $x,$
contradicting the definition of a solitary wave, which must be
non-constant. This completes the proof of Lemma \ref{lem eta sign standing}.
\end{proof}

A quick application of the maximum principle also yields
\begin{lem}\label{lem u sign standing}  Let $(u,\eta)$ be a solitary wave solution. Then
\begin{enumerate}[label=\rm(\alph*)]
\item\label{u symmetry} $(-u,\eta)$ is also a solitary wave solution.

\item\label{u positivity} If $u \ge 0$ then $u(x)>0.$

\item\label{u negativity} If $u \le 0$ then $u(x)<0.$ 
\end{enumerate}
\end{lem}
\begin{proof}
Part \ref{u symmetry} follows directly from \eqref{standing eqn}. For \ref{u positivity}, suppose that $u(x)\geq 0$ and that $u(x_{1})=0$ for some $x_{1}.$ Since $u$ is non-negative it must be the case
that $u^{\prime }(x_{1})=0.$ Thus, by uniqueness of the constant solution $%
(u,u^{\prime })=(0,0)$ of 
\begin{equation*}%\label{ge3sasef}
-a u^{\prime \prime }=u(1+\eta ),  
\end{equation*}%
we conclude that $u(x)\equiv 0$. In this case, $\eta ^{\prime \prime
}=\eta $ on $\R.$ But since the orbit is
homoclinic, $\eta $ must be bounded, which implies that $\eta \left(
x\right) =0$ for all $x$ as well, and this contradicts the definition
of solitary waves and Lemma \ref{lem eta sign standing}. Part \ref{u negativity} follows by a similar argument. 
\end{proof}

\begin{proof}[Proof of Proposition \ref{prop standing}]
The proof part \ref{standing soln 1} makes use of  function
\begin{equation*}%  \label{eq1ba}
h=u- \sqrt 2\eta,
\end{equation*}
which satisfies  
\begin{equation}\label{eq1bb}
\beta^2 h^{\prime\prime}=\left (1-{\frac {u}{\sqrt 2}}\right )h.
\end{equation}
If $(u, \eta)$ is a solitary wave solution with $u < \sqrt{2}$, then $1 - u/\sqrt{2} > 0$. Then from maximum principle we know that $h \equiv 0$.

Substituting $h=u-\sqrt 2\eta=0$ into~(\ref{eq1aav}) we get  
\begin{equation*}%  \label{eq1bbgtret}
\beta^2(\eta^{\prime})^{2}=\eta^{2}+{\frac {2}{3}}\eta^{3}.
\end{equation*}
which is the classical steady KdV equation. The solution is given by
\[
\eta_0(x) = -{3\over2} \mathrm{sech}^2\left({x\over 2\beta} \right), \quad \text{and hence} \quad u^-_0(x) = -{3\sqrt{2}\over2} \hbox{sech}^2\left({x\over 2\beta} \right).
\]

The proof of \ref{standing soln 2} makes use of the
functional $w=u+\sqrt{2}\eta ,$ which satisfies 
\begin{equation*}%\label{eq1bbff}
\beta^2w^{\prime \prime }=\left( 1+{\frac{u}{\sqrt{2}}}\right) w.  
\end{equation*} 
The remainder of the proof follows the same argument as before, with $w$ replacing $h.$ 
\end{proof}
%
%%%%%%%%%%%%%%%%%%%%%%%%%%%%%%%%%%%%%%%%%
%%%%%%%%%%%%%%%%%%%%%%%%%%%%%%%%%%%%%%%%%%%%%%%%%%%

%\noindent 
% {\bf Acknowledgements.}

%%%%%%%%%%%%%%%%%%%%%%%%%%%%%%%%%%%%%%%%%%%%%%%%%%%%%%%%%%%%
%%%%%%%%%%%%%%%%%%%%%%%%%%%%%%%%%%%%%%%%%%%%%%%%%%%%%%%%%%%%%%%%%%%%%%%%%%%%%%%

\end{document}